\documentclass[12pt,a4paper,reqno]{amsproc}
\usepackage{amsmath}
\usepackage{amssymb}
\usepackage{amsthm}
\usepackage[backrefs]{amsrefs}
\usepackage[german,danish,english]{babel}
\usepackage[utf8]{inputenc}
\usepackage[T1]{fontenc}
\usepackage{graphicx}
\usepackage{color}
\usepackage{mathrsfs}
\usepackage{comment}
\usepackage{multirow}
\usepackage{enumitem}
\usepackage[footnote,draft,danish,silent]{fixme}
\usepackage{MnSymbol}
\usepackage{hyphenat}

\numberwithin{equation}{section}

\newcommand{\NN}{\mathbb N}
\newcommand{\ZZ}{\mathbb Z}
\newcommand{\RR}{\mathbb R}
\newcommand{\QQ}{\mathbb Q}

\newcommand{\FF}{\mathbb F}

\DeclareMathOperator{\id}{id}

\DeclareMathOperator{\supp}{supp}

\DeclareMathOperator{\Cay}{Cay}
\DeclareMathOperator{\Sym}{Sym}
\DeclareMathOperator{\Stab}{Stab}

\DeclareMathOperator{\Ker}{Ker}
\DeclareMathOperator{\Sub}{Sub}
\DeclareMathOperator{\Prob}{Prob}
\DeclareMathOperator{\IRS}{IRS}
\DeclareMathOperator{\prodname}{prod}
\DeclareMathOperator{\GL}{GL}
\DeclareMathOperator{\SL}{SL}
\DeclareMathOperator{\stat}{stat}
\DeclareMathOperator{\BS}{BS}
\DeclareMathOperator{\gen}{gen}

\DeclareMathOperator{\findex}{f.i.}
\DeclareMathOperator{\fg}{f.g.}
\DeclareMathOperator{\pclosed}{p.c.}
\DeclareMathOperator{\almostn}{a.n.}

\newcommand{\calX}{\mathcal X}
\newcommand{\calG}{\mathcal G}
\newcommand{\calL}{\mathcal L}

\newcommand{\calC}{\mathcal C}
\newcommand{\calP}{\mathcal P}
\newcommand{\lla}{\langle \! \langle}
\newcommand{\rra}{\rangle \! \rangle}

\newtheorem{stn}{Sætning}[section]

\newtheorem{cor}[stn]{Corollary}

\newtheorem{prop}[stn]{Proposition}
\newtheorem{question}[stn]{Question}
\theoremstyle{definition}
\theoremstyle{remark}\newtheorem{rem}[stn]{Remark}

\theoremstyle{definition}\newtheorem{defn}[stn]{Definition}
\newtheorem{lem}[stn]{Lemma}
\newtheorem{thm}[stn]{Theorem}

\newcommand{\Pau}{P\u{a}unescu }

\title[Stability and Invariant Random Subgroups]{Stability and Invariant Random Subgroups}

\author[O.\ Becker]{Oren Becker}
\address{O.B., Hebrew University, Israel}
\email{oren.becker@mail.huji.ac.il}

\author[A.\ Lubotzky]{Alexander Lubotzky}
\address{A.L., Hebrew University, Israel}
\email{alexlub@math.huji.ac.il}

\author[A.\ Thom]{Andreas Thom}
\address{A.T., TU Dresden, Germany}
\email{andreas.thom@tu-dresden.de}

\begin{document}
\begin{abstract}
Consider $\Sym\left(n\right)$, endowed with the normalized Hamming metric
$d_n$. A finitely-generated group $\Gamma$ is \emph{P-stable} if every
almost homomorphism $\rho_{n_k}\colon \Gamma\rightarrow\Sym\left(n_k\right)$
(i.e., for every $g,h\in\Gamma$,
$\lim_{k\rightarrow\infty}d_{n_k}\left(
\rho_{n_k}\left(gh\right),\rho_{n_k}\left(g\right)\rho_{n_k}\left(h\right)\right)=0$)
is close to an actual homomorphism
$\varphi_{n_k} \colon\Gamma\rightarrow\Sym\left(n_k\right)$.
Glebsky and Rivera observed that finite groups are P-stable, while Arzhantseva
and P\u{a}unescu showed the same for abelian groups and raised many questions,
especially about P-stability of amenable groups.
We develop P-stability in general, and in particular for
amenable groups. Our main tool is the theory of invariant random subgroups (IRS),
which enables us to give a characterization of P-stability among amenable groups,
and to deduce stability and instability of various families of amenable groups.
\end{abstract}

\maketitle

\section{Introduction}

Let $\left(G_{n},d_{n}\right)_{n=1}^{\infty}$ be a sequence of groups
$G_{n}$ equipped with bi\hyp invariant metrics $d_{n}$, and let $\Gamma$
be a finitely-presented\footnote{The main results of this paper will be formulated and proved for general
finitely-generated groups, but for the simplicity of the exposition
in this introduction, we will assume that $\Gamma$ is finitely-presented.} group generated by a finite set $S=\left\{ s_{1},\dotsc,s_{m}\right\} $
subject to the relations $E=\left\{ w_{1},\dotsc,w_{r}\right\} \subset\FF$,
where $\FF$ is the free group on $S$.
In recent years, there has been some interest in the \emph{stability}
of $\Gamma$ with respect to $\calG=\left(G_{n},d_{n}\right)_{n=1}^{\infty}$
(cf. \cite{glebsky-rivera}, \cite{arzhantseva-paunescu},
\cite{DGLT}, and the references within), namely:
\begin{defn}
\label{def:intro-stable}
The group $\Gamma$ is \emph{stable} with respect to $\calG$ if
for every $\epsilon>0$ there exists $\delta>0$ such that if $\overline{g}=\left(g_{1},\dotsc,g_{m}\right)\in G_{n}^{m}$
satisfies $\sum_{i=1}^{r}d_{n}\left(w_{i}\left(\overline{g}\right),\text{id}_{G_{n}}\right)<\delta$,
then $\exists\overline{g}'=\left(g'_{1},\dots,g'_{m}\right)\in G_{n}^{m}$
with $\sum_{i=1}^{m}d_{n}\left(g_{i},g'_{i}\right)<\epsilon$ and
$w_{i}\left(\overline{g}'\right)=\id$ for every $i=1,\dotsc,m$
(i.e. $\overline{g}$ and $\overline{g}'$ are ``$\epsilon$-close'' and $\overline{g}'$ is a ``solution'' for $w_1=1,\dotsc,w_r=1$).
\end{defn}

In other words, every ``almost homomorphism'' from $\Gamma$ to
$G_{n}$ is close to an actual homomorphism. It is not difficult to
show (see \cite{arzhantseva-paunescu}) that the stability of $\Gamma$ with respect to $\left(G_{n},d_{n}\right)_{n=1}^{\infty}$
depends only on the group $\Gamma$, rather than the chosen presentation
- so the notion is well-defined.

The roots of this definition lie in some classical questions, asked
by Halmos, Turing, Ulam and others, whether ``almost solutions''
are always just a small deformation of precise solutions. The most
popular question of this sort, with origins in mathematical physics,
refers to the case where $G_{n}$ are some groups of matrices
and the question asks whether ``almost commuting matrices'' are ``near''
commuting matrices (which is the same as the stability of $\Gamma=\ZZ\times\ZZ$
defined above). The answer in this case depends very much on the metrics
$d_{n}$, e.g. if $G_{n}=U\left(n\right)$, the unitary groups, $\ZZ^{2}$
is stable with respect to the Hilbert-Schmidt norm, but not with respect
to the operator norm (cf. \cite{glebsky2010}, \cite{voiculescu}).
See also the introduction of \cite{arzhantseva-paunescu} for a short survey on
this problem.

In recent years (starting in \cite{glebsky-rivera} and in a more systematic
way in \cite{arzhantseva-paunescu}), there has been an interest in a discrete version,
i.e.\ the case $G_{n}=\Sym\left(n\right)$, the symmetric group on
$\left[n\right]=\left\{ 1,\dotsc,n\right\} $, where $d_{n}$ is the
normalized Hamming distance $d_{n}\left(\sigma,\tau\right)=\frac{1}{n}\cdot \left| \left\{ x\in\left[n\right]\mid\sigma\left(x\right)\neq\tau\left(x\right)\right\}\right| $.

We will refer to this as \emph{permutation stability} (or \emph{P-stability}
for short). One of the motivations to study this comes from ``local
testability'' of systems of equations in permutation groups (see
\cite{arzhantseva-paunescu}, \cite{glebsky-rivera}, \cite{becker-lubotzky}). Another motivation comes from
the hope to find a non-sofic group: As observed in \cite{glebsky-rivera},
non-residually-finite groups which are P-stable are \emph{not} sofic.
So it is desirable to have criteria for a group to be P-stable (see
\cite{DGLT}, for a similar strategy which led to the construction of
non-Frobenius-approximated groups).

But, as of now, there are very few methods and results proving P-stability
of groups. Clearly, free groups are P-stable, Glebsky and Rivera \cite{glebsky-rivera}
showed that finite groups are P-stable and Arzhantseva and \Pau
\cite{arzhantseva-paunescu} showed 
it for abelian groups. Now, a free product of P-stable groups is P-stable,
but this is not known, in general, for direct products!

In this paper, we develop P-stability and non-P-stability criteria.
Some of these are for general groups, but they are especially effective
for amenable groups.
Here is a sample of some conclusions of our work
(see Corollary \ref{cor:polycyclic},
Corollary \ref{cor:bs1n}
and Corollary \ref{cor:abels}):
\begin{thm}
\label{thm:intro-examples}~
\begin{enumerate}[label=(\roman*)]
\item Every polycyclic-by-finite group is P-stable.
\item For every $n\in\ZZ$, the Baumslag-Solitar group
$\BS\left(1,n\right)=\langle x,y\mid xyx^{-1}=y^{n}\rangle$ is P-stable.
\item There exists a finitely-presented solvable subgroup of $\GL_{4}\left(\QQ\right)$
which is not P-stable.
\end{enumerate}
\end{thm}
Part (i) of the theorem above answers a question raised
in \cite{konig-leitner-neftin} (the very special case of
the group $\BS\left(1,-1\right)$ was previously asked in
\cite{arzhantseva-paunescu}, see the
paragraph after Example 7.3 there).
Part (ii) completes the classification of P-stability of the Baumslag-Solitar groups $\BS\left(m,n\right)$ which was started in Example 7.3 of \cite{arzhantseva-paunescu}, except for the case $\left|m\right|=\left|n\right|\geq 2$.
Part (iii) shows that there is a finitely-presented amenable residually-finite group which is not P-stable, answering a question posed in \cite{arzhantseva-paunescu} (see, in \cite{arzhantseva-paunescu}, the paragraph before Theorem 7.2,  and Theorem 7.2iii).

\vspace{5mm}
The main novel method in the current paper is the use of the theory
of invariant random subgroups (IRS). This theory, which formally goes
back to the seminal work of Stuck and Zimmer \cite{stuck-zimmer}, got new life
in recent years starting with the work of Abert-Glasner-Virag \cite{agv}:
an IRS is defined as a $\Gamma$-invariant probability measure on the compact
space $\Sub(\Gamma)$ of all (closed) subgroups of $\Gamma$.
Let $\IRS(\Gamma)$ be the space of IRSs of $\Gamma$.
Fixing a surjective map $\pi:\FF\twoheadrightarrow\Gamma$
from the free group on $S$ onto $\Gamma$, we can think of $\IRS(\Gamma)$
as a subspace of $\IRS(\FF)$.

If $\Gamma$ is a discrete group, $\mu\in\IRS(\Gamma)$ will be called
a \emph{finite-index IRS} if it is atomic and all of its atoms are finite-index
subgroups of $\Gamma$. The IRSs which are a limit of the finite-index ones
are called \emph{co-sofic} (see \cite{gelander-icm}, Definition 15). We then prove
(see Theorem \ref{thm:main-theorem}):
\begin{thm}
\label{thm:intro-main}Let $\pi:\FF\twoheadrightarrow\Gamma$ be as before
and $\IRS(\Gamma)\subset\IRS(\FF)$.
\begin{enumerate}[label=(\roman*)]
\item If $\Gamma$ is P-stable and $\mu\in\IRS(\Gamma)$ is co-sofic
in $\FF$, then $\mu$ is co-sofic in $\Gamma$.
\item If $\Gamma$ is amenable, then $\Gamma$ is P-stable if and only if
every $\mu\in\IRS(\Gamma)$ is co-sofic (in $\Gamma$).
\end{enumerate}
\end{thm}

Theorem \ref{thm:intro-main}(ii) gives an ``if and only if'' criterion for
P-stability of amenable groups.
A crucial ingredient in the proof of this criterion is a result of Newman and Sohler
\cites{newman-sohler2011, newman-sohler2013}
which gives testability
of properties of hyperfinite families of graphs (see \cite{becker-lubotzky}
for more in this direction).
Actually, in the sequel, it will be more convenient for us to use Elek's treatment
\cite{elek2012} of the aforementioned theorem.
The amenability assumption in Theorem \ref{thm:intro-main}(ii) turns out to be
essential. Indeed, by \cite{becker-lubotzky-property-t}, the groups
$\SL_n(\ZZ)$ are not P-stable for $n\geq 3$, but as a corollary of the Stuck-Zimmer
Theorem \cite{stuck-zimmer},
all of their IRSs are supported on finite-index subgroups (and the trivial
subgroup $\{1\}$), and in particular they are co-sofic.

Let us sketch the argument for the $\Leftarrow$ direction of
Theorem \ref{thm:intro-main}(ii).
We think of the hypothesis that every IRS of $\Gamma$ is co-sofic as a
``density condition''.
Let $(X_n)_{n=1}^{\infty}$ be a sequence of finite
sets with almost actions of $\Gamma$, where $\Gamma$ is amenable. Then, as $n$ tends to infinity, $X_n$ converges to a p.m.p.\ action of $\Gamma$ and hence one obtains an $\IRS$ of $\Gamma$. Now, assuming the density condition, this $\IRS$ also arises as
a limit of finite $\Gamma$-actions. A little argument ensures that these actions can happen on the same sets $X_n$.
Both the sequence of actions and the sequence of almost actions are hyperfinite,
since the group $\Gamma$ is amenable.
Hence, by the Elek-Newman-Sohler result (on almost isomorphism of hyperfinite graphs with almost the same local statistics),
the almost actions are almost conjugate to the actions if $n$ is large enough -- end of the proof.
The role of the density condition in the argument above is to ensure that there are enough actions to model any possible $\IRS$ that could come up.
For the $\Rightarrow$ direction, any $\IRS$ (if $\Gamma$ is amenable) actually arises as a limit of almost actions of $\Gamma$ on finite sets, so that the condition of density of finite-index-$\IRS$ is also necessary.

In general, it is not easy to check
the criterion of Theorem \ref{thm:intro-main}(ii),
but if $\Gamma$ has only countably many subgroups
(see \cite{cutolo-smith} for a characterization of solvable groups with this property),
then every $\mu\in\IRS$$(\Gamma)$ is atomic and hence
supported only on almost-normal subgroups, i.e. subgroups $H$ for
which $\left[\Gamma:N_{\Gamma}\left(H\right)\right]<\infty$. 
This enables us to prove
(see Proposition \ref{prop:positive-application}):
\begin{thm}
\label{thm:application-positive}If $\Sub(\Gamma)$ is
countable and every almost-normal subgroup of $\Gamma$ is profinitely-closed
in $\Gamma$, then every $\mu\in\IRS(\Gamma)$ is co-sofic
in $\Gamma$, and if $\Gamma$ is also amenable, then $\Gamma$ is
P-stable.
\end{thm}

The first two points of Theorem \ref{thm:intro-examples} are deduced from
Theorem \ref{thm:application-positive}. We also show that if there
exists a \emph{finitely-generated} almost-normal subgroup of $\Gamma$
which is \emph{not} profinitely-closed, then $\Gamma$ is not P-stable (assuming that $\Gamma$ is amenable, but also
under a milder condition related to soficity), and this is used to prove part
(iii) of Theorem \ref{thm:intro-examples}.

The paper is organized as follows:
In Section \ref{sec:defs}, we give the definitions of P-stable equations and groups,
and explain the relation between the two notions.
In Sections \ref{sec:IRS} and \ref{sec:profinite},
we review the needed facts regarding
invariant random subgroups and the profinite topology, respectively.
In Section \ref{sec:borel}, we review the theories of hyperfinite actions
and graphs, and adapt the Newman-Sohler Theorem to our needs.
In Section \ref{sec:main-theorems}, we prove Theorem \ref{thm:intro-main}.
Finally, in Section \ref{sec:applications},
we use Theorem \ref{thm:intro-main} to prove Theorem \ref{thm:intro-examples}
and Theorem \ref{thm:application-positive}.

Let us end with saying that while our results give far reaching extensions
of the groups for which P-stability or non-P-stability is known,
we are still far from having the complete picture even for amenable
groups (or even for solvable groups). We still can not answer the
question whether, for given P-stable groups $\Gamma_1$
and $\Gamma_2$, $\Gamma_{1}\times\Gamma_{2}$ is also P-stable. Is
LERF a sufficient condition? More specifically, is the Grigorchuk group P-stable? etc.
(see Question \ref{question:amenable-lerf-is-stable}
and the discussion surrounding it).
Our work gives further motivation to understand and classify the IRS
of various finitely-generated groups.

\section*{Acknowledgements}
The authors would like to thank G\'{a}bor Elek for providing the proof of
Proposition \ref{prop:amenable-schreier-hyperfinite},
and to Benjy Weiss, Yair Glasner, Shahar Mozes and Tsachik Gelander for
valuable discussions.
This work is part of the PhD thesis of the first author at the Hebrew University of
Jerusalem.
The second author was supported in part by the ERC and the NSF.
The third author was supported in part by the ERC Consolidator Grant No.\ 681207. He thanks the Hebrew University for its hospitality during a visit in November 2017.
This research was partly done in the Israel Institute for Advanced Studies (IIAS) during the 2017-18 program on High Dimensional Combinatorics.

\section{Notation and Conventions}
\label{sec:notation}

Throughout the paper, we fix the following: Let $\Gamma$ be a finitely-generated
group. Present $\Gamma$ as a quotient of a finitely-generated free group
$\FF$ with quotient map $\pi\colon\FF\twoheadrightarrow \Gamma$. Fix
a finite basis $S=\left\{ s_{1},\dotsc,s_{m}\right\} $ for $\FF$.
Note that every result we prove for $\Gamma$ applies to $\FF$
as well as a special case by viewing $\FF$ as a quotient of
itself with $\pi$ being the identity map.

Recall that a $\Gamma$-set is a set $X$ endowed with an action of $\Gamma$, i.e. a homomorphism $\rho\colon \Gamma\rightarrow\Sym(X)$ called the \emph{structure homomorphism} of the action. When $\rho$ is understood from the context, we write $g\cdot x$ for $\rho(g)(x)$ where $g\in \Gamma$ and $x\in X$. We also write $\Gamma\curvearrowright X$ when we want to refer to an action of $\Gamma$ on a set $X$,
but suppress the structure homomorphism $\rho$. For a subgroup $H$ of $\Gamma$, we endow the coset space $\Gamma/H$, by default, with the action given by
$g\cdot(g_{1}H)=\left(gg_{1}\right)H$.

For a subset $A$ of $\Gamma$:
Write $A^{-1}=\left\{ a^{-1}\right\} _{a\in A}$ and $A^{\pm1}=A\cup A^{-1}$.
Write $\langle A\rangle$ for the subgroup generated by $A$ and
$\lla A\rra$ for the normal-closure of
$A$ in $\Gamma$ (i.e. the smallest normal subgroup of $\Gamma$ which contains
$A$, or, equivalently, the subgroup consisting of products of $\Gamma$-conjugates
of elements of $A^{\pm1}$).

For a $\Gamma$-set $X$:
Define a metric $d_{X}$ on $X$ where $d_{X}\left(x,y\right)$ is the
length, with respect to $S^{\pm1}$, of the shortest word $w\in\FF$
for which $w\cdot x=y$ (or $\infty$ if no such word exists,
but we shall always work within connected components anyway).
For an element $g\in \Gamma$ and a subset $A\subset \Gamma$, write $g\cdot A=\left\{ g\cdot a\mid a\in A\right\}$.
We use the notation $\coprod$ for disjoint unions, and write $X^{\coprod k}$
for the disjoint union of $k$ copies of $X$.

For a metric space $X$:
For an integer $r\geq0$ and a point $x\in X$, write $B_{X}\left(x,r\right)=\left\{ y\in X\mid d_{X}\left(x,y\right)\leq r\right\} $.
For an integer $r\geq0$ and a subset $A\subset X$, write $B_{X}\left(A,r\right)=\cup_{a\in A}B_{X}\left(a,r\right)$.
In case $X=\Gamma$, write $B_{\Gamma}\left(r\right)$ for $B_{\Gamma}\left(1_{\Gamma},r\right)$.

For a logical formula $\varphi$, we write ${\bf 1}_{\varphi}$ to
mean $1$ if $\varphi$ holds in the given context, and $0$ otherwise.
For a subgroup $H$ of $\Gamma$, write $H^{\Gamma}$ for the set of subgroups of $\Gamma$
which are conjugate to $H$, i.e. $H^\Gamma=\left\{ H^{g}\mid g\in \Gamma\right\}$.
Then, $\left|H^{\Gamma}\right|=\left[\Gamma:N_{\Gamma}\left(H\right)\right]$, and
we say that $H$ is \emph{almost-normal} in $\Gamma$ if $\left|H^{\Gamma}\right|<\infty$.
For an element $x$ in a measurable space $X$, we write $\delta_{x}$
for the Dirac measure at $x$. For $n\in\NN$, denote
$\left[n\right]=\left\{1,\dotsc,n\right\}$.

\section{Definitions}
\label{sec:defs}
\subsection{P-stable equations}

We refer to the elements of the basis $S$ of $\FF$ as \emph{letters},
and to the elements of $\FF$ as \emph{words}. For a word $w\in\FF$,
an integer $n\geq1$ and a tuple of permutations $\left(\sigma_{1},\dotsc,\sigma_{m}\right)\in\Sym\left(n\right)^{m}$,
we write $w\left(\sigma_{1},\dotsc,\sigma_{m}\right)$ for the element
of $\Sym\left(n\right)$ resulting from the substitution $s_{1}\mapsto\sigma_{1},\dotsc,s_{m}\mapsto\sigma_{m}$
applied to the word $w$. That is, if $w=s_{i_{1}}^{\epsilon_{1}}\cdots s_{i_{l}}^{\epsilon_{l}}$
for some integers $l\geq0$, $i_{1},\dotsc,i_{l}\in\left[m\right]$
and $\epsilon_{1},\dotsc,\epsilon_{l}\in\left\{ +1,-1\right\} $,
then $w\left(\sigma_{1},\dotsc,\sigma_{m}\right)=\sigma_{i_{1}}^{\epsilon_{1}}\cdots\sigma_{i_{l}}^{\epsilon_{l}}\in\Sym\left(n\right)$.
\begin{defn}
For $n\in\NN$, the \emph{normalized Hamming distance} $d_{n}$
on $\Sym\left(n\right)$ is defined by $d_{n}\left(\sigma_{1},\sigma_{2}\right)=\frac{1}{n}\left|\left\{ x\in\left[n\right]\mid\sigma_{1}\left(x\right)\ne\sigma_{2}\left(x\right)\right\} \right|$
where $\sigma_{1},\sigma_{2}\in\Sym\left(n\right)$.
\end{defn}

Note that $d_{n}$ is a bi-invariant metric on $\Sym\left(n\right)$.
\begin{defn}
Let $n\in\NN$, $E\subset\FF$ and $\left(\sigma_{1},\dotsc,\sigma_{m}\right)\in\Sym\left(n\right)^{m}$.
Then,
\begin{enumerate}[label=(\roman*)]
\item The tuple $\left(\sigma_{1},\dotsc,\sigma_{m}\right)$ is a \emph{solution}
for the system of equations $\left\{ w=1\right\} _{w\in E}$ if $w\left(\sigma_{1},\dotsc,\sigma_{m}\right)=1$
for each $w\in E$.
\item Assume that $E$ is a finite set. For $\delta>0$, the tuple $\left(\sigma_{1},\dotsc,\sigma_{m}\right)$
is a \emph{$\delta$-solution} for the system of equations
$\left\{ w=1\right\} _{w\in E}$ if
\[
\sum_{w\in E}d_{n}\left(w\left(\sigma_{1},\dotsc,\sigma_{m}\right),1\right)\leq\delta
\]
\end{enumerate}
\end{defn}

\begin{defn}
For $n\in\NN$ and $\overline{\sigma}=\left(\sigma_{1},\dotsc,\sigma_{m}\right),\overline{\tau}=\left(\tau_{1},\dotsc,\tau_{m}\right)\in\Sym\left(n\right)^{m}$,
define $d_{n}\left(\overline{\sigma},\overline{\tau}\right)=\sum_{i=1}^{m}d_{n}\left(\sigma_{i},\tau_{i}\right)$.
For $\epsilon>0$, if $d_{n}\left(\overline{\sigma},\overline{\tau}\right)\leq\epsilon$,
then we say that $\overline{\sigma}$ and
$\overline{\tau}$ are \emph{$\epsilon$-close}.
\end{defn}

\begin{defn}
\label{def:stable-equations}For $E\subset\FF$,
we say that the system of equations $\left\{ w=1\right\} _{w\in E}$
is \emph{stable in permutations} (or \emph{P-stable} for short) if for every $\epsilon>0$ there are $\delta>0$
and a finite subset $E_{0}\subset E$, such that for every $n\in\NN$
and $\delta$-solution $\left(\sigma_{1},\dotsc,\sigma_{m}\right)\in\Sym\left(n\right)^{m}$
for $\left\{ w=1\right\} _{w\in E_{0}}$,
there is a solution $\left(\tau_{1},\dotsc,\tau_{m}\right)\in\Sym\left(n\right)^{m}$
for $\left\{ w=1\right\} _{w\in E}$, such that $\left(\sigma_{1},\dotsc,\sigma_{m}\right)$
and $\left(\tau_{1},\dotsc,\tau_{m}\right)$ are $\epsilon$-close.
\end{defn}

\begin{rem}
\label{rem:def-for-equations-is-generalization}
The notion of a ``stable system'', introduced in \cite{arzhantseva-paunescu},
is a special case of Definition \ref{def:stable-equations}
for a \emph{finite} $E\subset\FF$. Indeed,
Definition 3.2 of \cite{arzhantseva-paunescu} says that a finite
$E\subset\FF$ is a \emph{stable system} if for every $\epsilon>0$,
there is $\delta>0$, such that every $\delta$-solution for
$\left\{ w=1\right\} _{w\in E}$
is $\epsilon$-close to a solution for
$\left\{ w=1\right\} _{w\in E}$. This is indeed equivalent
to our Definition \ref{def:stable-equations} in light of
Remark \ref{rem:delta-solution-basics} below.
\end{rem}


\begin{rem}
\label{rem:delta-solution-basics}~
For $E_{1}\subset E_{2}\subset\FF$,
every solution for $\left\{ w=1\right\} _{w\in E_{2}}$ is a solution
for $\left\{ w=1\right\} _{w\in E_{1}}$.
Moreover, assuming that $E_{2}$ is finite, for $\delta>0$, every $\delta$-solution
for $\left\{ w=1\right\} _{w\in E_{2}}$ is a $\delta$-solution for
$\left\{ w=1\right\} _{w\in E_{1}}$.
If $w_{1},w_{2}\in\FF$, then every simultaneous solution for
$\left\{ w_{1}=1\right\} $ and $\left\{ w_{2}=1\right\} $ is a solution
for $\left\{ w_{1}\cdot w_{2}=1\right\} $.
If $t,w\in\FF$ , then every solution for $\left\{ w=1\right\} $
is a solution for $\left\{ t\cdot w\cdot t^{-1}=1\right\} $.
By the above, if $E\subset\FF$ and $n\in\NN$, then
a tuple $\left(\sigma_{1},\dotsc,\sigma_{m}\right)\in\Sym\left(n\right)^{m}$
is a solution for $\left\{ w=1\right\} _{w\in E}$ if and only
if it is a solution for $\left\{ w=1\right\} _{w\in\lla E\rra}$.
\end{rem}

\begin{lem}
\label{lem:normal-closure-delta-solution}Let $E\subset\FF$.
Take $\tilde{\delta}>0$ and a finite subset $\tilde{E}_{0}\subset\lla E\rra$.
Then, there are $\delta>0$ and a finite subset $E_{0}\subset E$
, such that every $\delta$-solution for $\left\{ w=1\right\} _{w\in E_{0}}$
is a $\tilde{\delta}$-solution for $\left\{ w=1\right\} _{w\in\tilde{E}_{0}}$.
\end{lem}

\begin{proof}
For every $w\in\tilde{E}_{0}$, write
$w=\prod_{i=1}^{l_{w}}t_{w,i}\cdot q_{w,i}^{\epsilon_{w,i}}\cdot t_{w,i}^{-1}$
where $l_{w}\geq0$, $\left\{ q_{w,i}\right\} _{i=1}^{l_{w}}\subset E$,
$\left\{\epsilon_{w,i}\right\}_{i=1}^{l_w}\subset\left\{1,-1\right\}$
and $\left\{ t_{w,i}\right\} _{i=1}^{l_{w}}\subset\FF$. Let
$E_{0}=\left\{ q_{w,i}\mid w\in\tilde{E}_{0},1\leq i\leq l_{w}\right\} $
and $C=\sum_{w\in\tilde{E}_{0}}l_{w}$. Define $\delta=\frac{1}{C}\cdot\tilde{\delta}$.
Take $n\in\NN$ and a $\delta$-solution $\left(\sigma_{1},\dotsc,\sigma_{m}\right)\in\Sym\left(n\right)^{m}$
for $\left\{ w=1\right\} _{w\in E_{0}}$. For every $x\in\left[n\right]$
and $w\in\FF$, write $w\cdot x$ for $w\left(\sigma_{1},\dotsc,\sigma_{m}\right)\left(x\right)$.
For every $w\in E_{0}$, write $F_{w}=\left\{ x\in\left[n\right]\mid w\cdot x\ne x\right\} $.
Then, $\sum_{w\in E_{0}}\left|F_{w}\right|\leq\delta n$. A fortiori,
$\left|F_{w}\right|\leq\delta n$ for each $w\in E_{0}$. 

Let $w\in\tilde{E}_{0}$. Define $P_{w}=\cup_{i=1}^{l_{w}}t_{w,i}\cdot F_{q_{w,i}}$.
Then $\left|P_{w}\right|\leq\sum_{i=1}^{l_{w}}\left|F_{q_{w,i}}\right|\leq l_{w}\cdot\delta n$.
For $x\in\left[n\right]$ and $1\leq i\leq l_{w}$, if $t_{w,i}^{-1}\cdot x\notin F_{q_{w,i}}$,
\[
\left(t_{w,i}\cdot q_{w,i}\cdot t_{w,i}^{-1}\right)\cdot x=t_{w,i}\cdot q_{w,i}\cdot\left(t_{w,i}^{-1}\cdot x\right)=t_{w,i}\cdot\left(t_{w,i}^{-1}\cdot x\right)=x\text{ ,}
\]
and so if $x\notin P_{w}$,
\[
w\cdot x=\left(\prod_{i=1}^{l_{w}}t_{w,i}\cdot q_{w,i}\cdot t_{w,i}^{-1}\right)\cdot x=x
\]

Therefore, $d_{n}\left(w\left(\sigma_{1},\dotsc,\sigma_{m}\right),1\right)\leq\frac{1}{n}\cdot\left|P_{w}\right|\leq\frac{1}{n}\cdot l_{w}\cdot\delta n=l_{w}\cdot\delta$.
Finally, $\sum_{w\in\tilde{E}_{0}}d_{n}\left(w\left(\sigma_{1},\dotsc,\sigma_{m}\right),1\right)\leq C\cdot\delta=\tilde{\delta}$.
In other words, $\left(\sigma_{1},\dotsc,\sigma_{m}\right)$ is a
$\tilde{\delta}$-solution for $\left\{ w=1\right\} _{w\in\tilde{E}_{0}}$.
\end{proof}
\begin{lem}
\label{lem:normal-closure-equivalence}Let $E\subset\FF$.
Then, $\left\{ w=1\right\} _{w\in E}$ is P-stable if and
only if $\left\{ w=1\right\} _{w\in\lla E\rra}$
is P-stable.
\end{lem}

\begin{proof}
Assume that $\left\{ w=1\right\} _{w\in E}$ is P-stable. Let $\epsilon>0$.
Then, there is $\delta>0$ and a finite subset $E_0\subset E$ (and so $E_0\subset \lla E\rra$) such that every $\delta$-solution
for $\{w=1\}_{w\in E_{0}}$ is $\epsilon$-close to a solution for $\{w=1\}_{w\in E}$. The latter is a solution for $\{w=1\}_{w\in\lla E\rra}$ as well by Remark \ref{rem:delta-solution-basics}, and so $\left\{ w=1\right\} _{w\in \lla E\rra}$ is P-stable.

The reverse implication follows similarly using Lemma \ref{lem:normal-closure-delta-solution}.
\end{proof}

\subsection{P-stable groups}

Each $\Gamma$-set $X$ is naturally an $\FF$-set. Conversely,
for an $\FF$-set $X$, if the structure homomorphism $\rho\colon\FF\rightarrow\Sym\left(X\right)$
factors through $\Gamma$ by $\pi\colon\FF\twoheadrightarrow\Gamma$, then $X$ is naturally a $\Gamma$-set.
This condition is equivalent to the following: for every $w\in\FF$
and $x\in X$, if $\pi\left(w\right)=1_{\Gamma}$, then $w\cdot x=x$.
\begin{defn}
For $\delta>0$ and a finite subset $E_{0}\subset\Ker\pi$,
a finite $\FF$-set $X$ is a \emph{$\left(\delta,E_{0}\right)$-almost-$\Gamma$-set}
if $\sum_{w\in E_{0}}\Pr_{x\in X}\left(w\cdot x\neq x\right)\leq\delta$
(where $X$ is endowed with the uniform distribution).
\end{defn}

\begin{defn}
\label{def:gen-metric}
Let $X$ and $Y$ be finite $\FF$-sets of the same cardinality. For a bijection $f\colon X\rightarrow Y$,
define 
\[
\|f\|_{\gen}=\frac{1}{\left|S\right|}\cdot
\sum_{s\in S} \Pr_{x\in X}\left(f\left(s\cdot x\right)\neq s\cdot f\left(x\right)\right) \text{.}
\]
Finally, 
\[
d_{\gen}\left(X,Y\right)=\min\left\{ \|f\|_{\gen}\mid\text{\ensuremath{f\colon X\rightarrow Y} is a bijection}\right\} 
\text{.}\]
We refer to $d_{\gen}$ as the \emph{generator-metric}.
\end{defn}
Definition $\ref{def:gen-metric}$ will be generalized by
Definition $\ref{def:gen-metric-probspace}$.

For $n\in\NN$ and a tuple
$\overline{\sigma}=\left(\sigma_{1},\dotsc,\sigma_{m}\right)\in\Sym\left(n\right)^{m}$,
write $\FF\left(\overline{\sigma}\right)$ for
the $\FF$-set whose point set is $\left[n\right]$, with the
action given by $s_{i}\cdot x=\sigma_{i}\left(x\right)$ for each
$1\leq i\leq m$.
Note that for
$\overline{\sigma},
\overline{\tau}
\in\Sym\left(n\right)^{m}$, 
$\|\id\|_{\gen}=d_n\left(\overline{\sigma},
\overline{\tau}\right)$,
where $\id\colon\FF\left(\overline{\sigma}\right)
\rightarrow \FF\left(\overline{\tau}\right)$
is the identity map $\left[n\right]\rightarrow\left[n\right]$.

In Definition \ref{def:stable-equations}, we generalized (see Remark \ref{rem:def-for-equations-is-generalization}) the notion of a P-stable system 
of equations from finite systems (as studied in \cite{glebsky-rivera} and \cite{arzhantseva-paunescu}), to possibly infinite systems. Analogously, we now generalize the notion of P-stable groups, studied in the aforementioned papers for finitely-presented groups and coinciding with Definition \ref{def:intro-stable} in the introduction, to finitely-generated groups.
\begin{defn}
\label{def:stable-group}The group $\Gamma$ is \emph{stable in permutations} 
(or \emph{P-stable} for short) 
if for every $\epsilon>0$ there are $\delta>0$ and a finite subset
$E_{0}\subset\Ker\pi$, such that for every finite $\FF$-set
$X$, if $X$ is a $\left(\delta,E_{0}\right)$-almost-$\Gamma$-set, then
there is a $\Gamma$-set $Y$ such that $\left|X\right|=\left|Y\right|$
and $d_{\gen}\left(X,Y\right)\leq\epsilon$.
\end{defn}

\begin{lem}
The group $\Gamma$ is P-stable if and only if the system of equations
$\left\{ w=1\right\} _{w\in\Ker\pi}$ is P-stable.
\end{lem}

\begin{proof}
Assume that $\Gamma$ is P-stable. Then, for $\epsilon>0$, there
are $\delta>0$ and $E_{0}\subset\Ker\pi$ satisfying the condition
in Definition \ref{def:stable-group}. Let $\overline{\sigma}\in\Sym\left(n\right)^{m}$
be a $\delta$-solution for $\left\{ w=1\right\} _{w\in E_{0}}$.
Let $X=\FF\left(\overline{\sigma}\right)$. Then,
$X$ is a $\left(\delta,E_{0}\right)$-almost-$\Gamma$-set. Therefore,
there is a $\Gamma$-set $Y$ and a bijection $f\colon X\rightarrow Y$ satisfying
$\|f\|_{\gen}\leq\epsilon$. Define a tuple
$\overline{\tau}=\left(\tau_{1},\dotsc,\tau_{m}\right)\in\Sym\left(n\right)^{m}$
by $\tau_{i}\left(x\right)=f^{-1}\left(s_{i}\cdot f\left(x\right)\right)$.
Then, $\overline{\tau}$ is a solution for $E$
and it is $\epsilon$-close to $\overline{\sigma}$.

In the other direction, assume that $\left\{ w=1\right\} _{w\in\Ker\pi}$
is P-stable. Then, for every $\epsilon>0$, there are $\delta>0$
and $E_{0}\subset\Ker\pi$ satisfying the condition in Definition
\ref{def:stable-equations}. Let $X$ be a $\left(\delta,E_{0}\right)$-almost-$\Gamma$-set.
Denote $\left|X\right|=n$, take an arbitrary bijection $f\colon\left[n\right]\rightarrow X$,
and define a tuple $\overline{\sigma}=\left(\sigma_{1},\dotsc,\sigma_{m}\right)\in\Sym\left(n\right)^{m}$
by $\sigma_{i}\left(x\right)=f^{-1}\left(s_{i}\cdot f\left(x\right)\right)$.
Then, $\overline{\sigma}$ is a $\delta$-solution
for $E_{0}$. Therefore, there is a solution $\overline{\tau}\in\Sym\left(n\right)^m$
for $\Ker\pi$ which is $\epsilon$-close to $\overline{\sigma}$.
Let $Y=\FF\left(\overline{\tau}\right)$. Consider
$f$ as a function from $Y$ to $X$. Then, $\|f\|_{\gen}\leq\epsilon$, and
so $d_{\gen}\left(X,Y\right)\leq\epsilon$.
\end{proof}

\begin{rem}
Definition \ref{def:stable-group} introduces the notion
of a P-stable group $\Gamma$ using a given presentation of $\Gamma$
as a quotient of a finitely-generated free group. Nevertheless, the
definition depends only on $\Gamma$ as an abstract group. Indeed, consider
two finitely-generated free groups $\FF_{S}$ and $\FF_{T}$
with bases $S$ and $T$, respectively. Denote the generator-metrics on finite
$\FF_{S}$-sets and on finite $\FF_{T}$-set by $d_{\gen}^{S}$
and $d_{\gen}^{T}$, respectively. Present $\Gamma$ in two ways: $\pi_{S}\colon\FF_{S}\twoheadrightarrow \Gamma$
and $\pi_{T}\colon\FF_{T}\twoheadrightarrow\Gamma$. For every $t\in T$, let
$v_{t}\in\FF_{S}$ be a word for which $\pi_{S}\left(v_{t}\right)=\pi_{T}\left(t\right)$.
Define a homomorphism $\alpha\colon\FF_{T}\rightarrow\FF_{S}$
by extending the law $\alpha\left(t\right)=v_{t}$. Then, every $\FF_{S}$-set
is naturally an $\FF_{T}$-set. There is a constant $C>0$
such that for every pair $X$ and $Y$ of finite $\FF_{S}$-sets
of the same cardinality, $d_{\gen}^{T}\left(X,Y\right)\leq C\cdot d_{\gen}^{S}\left(X,Y\right)$.
Moreover, for $E_{0}\subset\Ker\pi_{T}$, $\delta>0$ and an
$\FF_{S}$-set $X$, if $X$ is a
$\left(\delta,\alpha\left(E_{0}\right)\right)$-almost-$\Gamma$-set,
then as an $\FF_{T}$-set it is a
$\left(\delta,E_{0}\right)$-almost-$\Gamma$-set.
Running the same arguments with $S$ and
$T$ reversed, we see that $\Gamma$ is P-stable with respect to
$\pi_{S}$ if and only if it is P-stable with respect to $\pi_{T}$.
More concisely, we have shown that the metrics $d_{\gen}^{S}$ and $d_{\gen}^{T}$
are bi-Lipschitz equivalent and that the notions of almost-$\Gamma$-sets with respect
to $\FF_S$ and to $\FF_T$ are essentially equivalent.
\end{rem}

\section{Invariant random subgroups}
\label{sec:IRS}

We recall the notion of an \emph{invariant random subgroup} (a.k.a. IRS, see \cite{agv, gelander-notes,gelander-icm}).
Write $2^{\Gamma}$ for the set of functions $f\colon\Gamma\rightarrow\left\{0,1\right\}$, and identify $2^{\Gamma}$ with the power set of $\Gamma$ by associating each function $f\colon\Gamma\rightarrow\left\{0,1\right\}$ with the set $\left\{ g \in \Gamma \mid f \left( g \right) = 1 \right\}$.
Denote the set of subgroups of $\Gamma$ by $\Sub(\Gamma)$.
Endow $\Sub(\Gamma)$ with the Chabauty topology, which,
for discrete groups, is just the subspace topology induced from the
product topology on $2 ^{\Gamma}$ and the inclusion
$\Sub(\Gamma)\subset2 ^{\Gamma}$. A sequence
$\left(U_{n}\right)_{n=1}^{\infty}$ in $2 ^{\Gamma}$ converges
if and only if for every $g\in \Gamma$, either $g\in U_{n}$ for all large
enough $n$, or $g\notin U_{n}$ for all large enough $n$. In this
case, $\limsup U_{n}=\liminf U_{n}=U$, where $U$ is the limit of
of the sequence. This also shows that $\Sub(\Gamma)$
is a closed subspace of $2 ^{\Gamma}$, and so it is
compact.
The group $\Gamma$ acts on $\Sub(\Gamma)$ continuously by conjugation.
Write $\Sub_{\fg}(\Gamma)$,
$\Sub_{\findex}(\Gamma)$ and
$\Sub_{\almostn}(\Gamma)$ for the subspaces of
$\Sub(\Gamma)$ of finitely-generated subgroups, finite-index subgroups and almost-normal subgroups,
respectively.

For an element $w\in \Gamma$, define $C_{w}=\left\{ H\leq \Gamma\mid w\in H\right\} $.
For an integer $r\geq0$ and a subset $W\subset \Gamma$, let $C_{r,W}=\left\{ H\leq \Gamma\mid H\cap B_{\Gamma}\left(r\right)=W\cap B_{\Gamma}\left(r\right)\right\} $.
Note that such sets $C_{w}$ and $C_{r,W}$ are clopen in $\Sub(\Gamma)$,
and so their characteristic functions are continuous. For a given
subgroup $K\leq \Gamma$, the subspace $\left\{ H\leq \Gamma\mid K\leq H\right\} $
of $\Sub(\Gamma)$ is closed since it equals $\cap_{k\in K}C_{k}$.

We exhibit a metric generating the topology of $\Sub(\Gamma)$.
Fix an enumeration $\left(g_{i}\right)_{i=1}^{\infty}$ of the elements
of $\Gamma$. The metric on $2 ^{\Gamma}$ defined by $d_{\prodname}\left(U_{1},U_{2}\right)=\sum_{i=1}^{\infty}2^{-i}\cdot{\bf 1}_{U_{1}\cap\left\{ g_{i}\right\} =U_{2}\cap\left\{ g_{i}\right\} }$
induces the product topology on $2 ^{\Gamma}$, and so
its restriction to $\Sub(\Gamma)$ induces the Chabauty topology.
Note that for every $\epsilon>0$, there is an integer $r\geq 1$,
such that for all $H_1,H_2\leq\Gamma$, if
$H_1\cap B_{\Gamma}\left(r\right)=H_2\cap B_{\Gamma}\left(r\right)$,
then $d_{\prodname}\left(H_1,H_2\right)<\epsilon$.

Consider the space $\Prob\left(\Sub(\Gamma)\right)$
of Borel regular probability measures on $\Sub(\Gamma)$.
We shall refer to elements of $\Prob\left(\Sub(\Gamma)\right)$ as
\emph{random-subgroups}.
The group $\Gamma$ acts on $\Prob\left(\Sub(\Gamma)\right)$
by conjugation, i.e. $\left(g\cdot\mu\right)\left(A\right)=\mu\left(g^{-1}Ag\right)$.
We write $\IRS(\Gamma)$ for the subspace of
$\Prob\left(\Sub(\Gamma)\right)$
of conjugation-invariant random subgroups, namely $\IRS(\Gamma)=
\Prob\left(\Sub(\Gamma)\right)^{\Gamma}$.
We shall refer to elements of $\IRS(\Gamma)$
as \emph{invariant random-subgroups}, or \emph{IRSs}.
Endow $\IRS(\Gamma)$ with the weak-$*$ topology. A sequence
$\left(\mu_{n}\right)_{n=1}^{\infty}$ in $\IRS(\Gamma)$ converges
in the weak-$*$ topology to $\mu\in\IRS(\Gamma)$ if and only
if $\int fd\mu_{n}\rightarrow\int fd\mu$ for every continuous function
$f\colon\Sub(\Gamma)\rightarrow\RR$. It follows from the Riesz-Markov
and Banach-Alaoglu theorems that $\IRS(\Gamma)$ is a compact space.
Moreover, under the weak-$*$ topologies, $\IRS(\Gamma)$ is metrizable
by the L\'{e}vy-Prokhorov metric.
We shall only use the metrizability of $\IRS(\Gamma)$
to identify the closure of a given subset $A$ of $\IRS(\Gamma)$
with the set of limits of convergent sequences (rather than nets) in $A$.

The space $\Sub(\Gamma)$ enjoys a useful sequence
$\calP_n(\Gamma)$
of partitions into finitely many clopen sets.
For $n\in\NN$, define the partition
\[
\calP_n=\calP_n(\Gamma)=\left\{C_{n,W}\mid
W\subset B_{\Gamma}\left(n\right),
C_{n,W}\neq\emptyset\right\}
\text{ .}
\]
For a continuous function
$f\colon\Sub(\Gamma)\rightarrow\RR$,
define sequence of continuous functions
$f_n\colon\Sub(\Gamma)\rightarrow\RR$
by $f_n=\sum_{A\in\calP_{n}}f\left(K_A\right)\cdot{\bf 1}_{A}$,
where $K_A$ is an arbitrary element of $A$ for each $A\in\calP_n$.
Then, since $f$ is uniformly continuous,
$\left\|f_n-f\right\|_{\infty}\rightarrow 0$.
This shows that for $\mu\in\IRS(\Gamma)$ and a
sequence $\left(\mu_n\right)_{n=1}^{\infty}$ in
$\IRS(\Gamma)$, $\mu_n\rightarrow\mu$
in the weak-$*$
topology if and only if for every integer $r\geq 1$ and
$W\subset B_{\Gamma}\left(r\right)$,
$\mu_{n}\left(C_{r,W}\right)
\overset{n\rightarrow\infty}{\longrightarrow}
\mu\left(C_{r,W}\right)$.
%

Let $\mu\in\IRS(\Gamma)$ be an atomic IRS.
Then, all atoms of $\mu$ must be almost-normal subgroups of $\Gamma$.
Fix $n\in\NN$. Take pairwise non-conjugate subgroups
$H_{1},\dots,H_{k}$ of $\Gamma$
such that $M>1-\frac{1}{n}$, where $M=\sum_{i=1}^{k}m_{i}$ and $m_{i}=\mu\left(H_{i}^{\Gamma}\right)$.
Write $A=\cup_{i=1}^{k}H_i^{\Gamma}$.
Let $\mu_{n}\in\IRS(\Gamma)$
be the atomic IRS assigning measure
$\frac{m_i}{M\cdot\left|H_i^{\Gamma}\right|}$ to each conjugate of $H_i$
for every $1\leq i\leq k$.
Take a continuous function $f\colon\Sub(\Gamma)\rightarrow\RR$.
Then,
\begin{align*}
\left|\int fd\mu_{n}-\int fd\mu\right|\le & \left|\int_{A}fd\mu_{n}-\int_{A}fd\mu\right|+\\
 & \left|\int_{\Sub(\Gamma)\setminus A}fd\mu_{n}-\int_{\Sub(\Gamma)\setminus A}fd\mu\right|\\
\leq & \sum_{i=1}^{k}\frac{1}{\left|H_i^{\Gamma}\right|}\cdot
\sum_{K\in H_i^{\Gamma}}
\left(\frac{m_{i}}{M}-m_{i}\right)\cdot f\left(K\right)+\\
& \left|0-\int_{\Sub(\Gamma)\setminus A}fd\mu\right|\\
\leq & \left(\frac{1}{M}-1\right)\cdot\left(\sum_{i=1}^{k}m_{i}\right)
\cdot\|f\|_{\infty}
+\frac{1}{n}\cdot\|f\|_{\infty}\\
= & \left(\left(1-M\right)+\frac{1}{n}\right)\cdot\|f\|_{\infty}\\
\leq & \frac{2}{n}\cdot\|f\|_{\infty}
\end{align*}
and so $\mu_{n}\rightarrow\mu$.
We have thus shown that every atomic IRS $\mu\in\IRS(\Gamma)$
is the limit of a sequence
$\left(\mu_n\right)_{n=1}^{\infty}$ of finitely-supported atomic IRSs with
$\supp\left(\mu_n\right)\subset\supp\left(\mu\right)$.

Recall that a standard Borel space is a measurable space which is isomorphic,
as a measurable space, to a compact metric space with its Borel $\sigma$-algebra.
\begin{defn}
A \emph{probability space} is a standard Borel space endowed with a Borel regular
probability measure.
A \emph{$\Gamma$-probability-space} $X$ is a probability space endowed with
a Borel action $\Gamma\curvearrowright X$.
If the action is probability measure preserving (p.m.p.), we say for short that
that $X$ is a \emph{p.m.p. $\Gamma$-space}.
\end{defn}

Let $\left(X,\nu\right)$ be a $\Gamma$-probability-space.
Then, the stabilizer map $f\colon X\rightarrow\Sub(\Gamma)$
defined by $f\left(x\right)=\Stab_{\Gamma}\left(x\right)$ is a Borel map,
and so we may define the pushforward measure
$\mu=f_{*}\nu\in\Prob\left(\Sub(\Gamma)\right)$.
By definition, $\mu\left(A\right)=\nu\left(f^{-1}\left(A\right)\right)$
for every Borel set $A\subset\Sub(\Gamma)$.
If $X$ is a p.m.p. $\Gamma$-space, then $\mu\in\IRS(\Gamma)$,
and we refer to $\mu$ as the
\emph{IRS associated with $X$}.

For a sequence $\left(X_{n}\right)_{n=1}^{\infty}$ of 
p.m.p. $\Gamma$-spaces
with associated sequence of IRSs $\left(\mu_{n}\right)_{n=1}^{\infty}$,
if $\mu_{n}\rightarrow\mu$, then we say that $\mu$ is the \emph{limiting
IRS} of $\left(X_{n}\right)_{n=1}^{\infty}$.

We shall consider both $\IRS(\Gamma)$ and $\IRS(\FF)$.
The discussion above applies to $\IRS(\FF)$ as a special case.
We identify $\IRS(\Gamma)$ with the subspace of $\IRS(\FF)$
of measures supported on subgroups which contain $\Ker\pi$. With
this identification, $\IRS(\Gamma)$
is a closed subspace of $\IRS(\FF)$.

\begin{defn}
A random-subgroup
$\mu\in
\Prob\left(\Sub(\Gamma)\right)$ is a
\emph{finite-index random-subgroup} if it is atomic and all of
its atoms are finite-index subgroups of $\Gamma$.
Write $\IRS_{\findex}(\Gamma)$
for the subspace of $\IRS(\Gamma)$
consisting of the finite-index IRSs.
\end{defn}
\begin{defn}
An IRS $\mu\in\IRS(\Gamma)$
is co-sofic if it is the weak-$*$ limit of a sequence 
$\left(\mu_{n}\right)_{n=1}^{\infty}$
in $\IRS(\Gamma)$ 
of finite-index IRSs.
\end{defn}
By the discussion above regarding approximation of
atomic IRSs by finitely-supported atomic IRSs,
and since $\IRS(\Gamma)$ is metrizable,
an IRS $\mu\in\IRS(\Gamma)$
is co-sofic if and only if it is the limit of a sequence of
\emph{finitely-supported} finite-index IRSs.

\begin{lem}
\label{lem:approximation-of-cosofic-by-actions}Let $\mu\in\IRS(\Gamma)$
be a co-sofic IRS. Then, there is a sequence $\left(X_{n}\right)_{n=1}^{\infty}$
of finite $\Gamma$-sets whose associated sequence of IRSs $\left(\mu_{n}\right)_{n=1}^{\infty}$
converges to $\mu$.
\end{lem}

\begin{proof}
Let $\left(\mu_{n}\right)_{n=1}^{\infty}$ be a sequence in $\IRS(\Gamma)$
of finitely-supported finite-index IRSs converging to $\mu$.
Fix $n\in\NN$. Take pairwise
non-conjugate subgroups $H_{1},\dotsc,H_{k}$ such that $\mu_{n}$ is
supported on $\cup_{i=1}^{k}H_{i}^{\Gamma}$. Take positive integers $l_{1},\dotsc,l_{k}$
satisfying $\left|\frac{l_{i}}{S}-m_{i}\right|<\frac{1}{kn}$, where
$S=\sum_{i=1}^{k}l_{i}$ and $m_{i}=\mu\left(H_{i}^{\Gamma}\right)$. Let
$X_{n}=\coprod_{i=1}^{k}\left(\Gamma/H_{i}\right)^{\coprod l_{i}}$ and
write $\nu_{n}\in\IRS(\Gamma)$ for the IRS associated with
$X_{n}$. Take a continuous function $f\colon\Sub(\Gamma)\rightarrow\RR$.
Then,
\begin{align*}
\left|\int fd\nu_{n}-\int fd\mu_{n}\right|= & \left|\sum_{i=1}^{k}\left(\frac{l_{i}}{S}-m_{i}\right)\cdot\left(\frac{1}{\left|H_{i}^{\Gamma}\right|}\cdot\sum_{K\in H_{i}^{\Gamma}}f\left(K\right)\right)\right|\\
\leq & \sum_{i=1}^{k}\frac{1}{kn}\cdot\|f\|_{\infty}\\
= & \frac{1}{n}\cdot\|f\|_{\infty}
\end{align*}
and so $\nu_{n}\rightarrow\mu$.
\end{proof}
\begin{lem}
\label{lem:co-sofic-by-support}
Let $\mu\in\IRS(\Gamma)$
be a co-sofic IRS.
Then, $\supp\left(\mu\right)\subset \overline{\Sub_{\findex}(\Gamma)}$.
\end{lem}
\begin{proof}
Let $H\in\supp\left(\mu\right)$.
Take a sequence $\left(\mu_n\right)_{n=1}^{\infty}$ of
finite-index IRSs converging to $\mu$.
Let $r\in\NN$.
Since $\mu\left(C_{r,H}\right)>0$, there is $n\in\NN$
for which $\mu_n\left(C_{r,H}\right)>0$.
Therefore, there is a finite-index subgroup $H_r$ of $\Gamma$
satisfying $H_r\in C_{r,H}$.
Then $H_r\rightarrow H$, and so
$H\in\overline{\Sub_{\findex}(\Gamma)}$.
\end{proof}

Given an IRS
$\mu\in\IRS(\Gamma)$, we say that $\mu$ is \emph{co-sofic in $\FF$} if it is co-sofic as an element of
$\IRS(\FF)$
under the natural inclusion of $\IRS(\Gamma)$ in $\IRS(\FF)$, i.e., if it is the limit of a sequence $\left(\mu_{n}\right)_{n=1}^{\infty}$
in $\IRS(\FF)$ of finite-index IRSs.
Therefore, for emphasis, we sometimes say ``co-sofic in
$\Gamma$'' instead of ``co-sofic''.

\section{Remarks on the profinite topology on an abstract group}
\label{sec:profinite}
Recall that the profinite topology on the group $\Gamma$ is the topology,
making $\Gamma$ a topological group, for which the finite-index
subgroups form a basis of neighborhoods of $1_{\Gamma}$. The closure of
a subgroup $H$ of $\Gamma$ under the profinite topology of $\Gamma$ equals
the intersection of the finite-index subgroups of $\Gamma$ containing
$H$. We refer to this closure as \emph{the profinite closure} $\overline{H}$
of $H$ in $\Gamma$, and if $\overline{H}=H$, we say that $H$ is \emph{profinitely-closed} in $\Gamma$.
If $H$ is normal in $\Gamma$, then $\overline{H}$ equals the intersection
of the \emph{normal} finite-index subgroups of $\Gamma$ which contain $H$.
Write $\Sub_{\pclosed}(\Gamma)$
for the subspace of $\Sub(\Gamma)$
of profinitely-closed subgroups of $\Gamma$.
Note that the trivial subgroup $\left\{1\right\}$ of $\Gamma$ is profinitely-closed
if and only if $\Gamma$ is residually-finite.
\begin{lem}
\label{lem:profinite-vs-chabauty}
$\overline{\Sub_{\findex}(\Gamma)}
\cap\Sub_{\fg}(\Gamma)
\subset \Sub_{\pclosed}(\Gamma)
\subset
\overline{\Sub_{\findex}(\Gamma)}
\text{ .}$
\end{lem}
\begin{proof}
For the right inclusion, take
$H\in\Sub_{\pclosed}(\Gamma)$.
Then, there is a sequence $\left(H_n\right)_{n=1}^{\infty}$
of finite-index
subgroups of $\Gamma$ such that
$H=\cap_{n=1}^{\infty}H_n$. Hence, $H_n\rightarrow H$.

For the left inclusion, take
$H\in\overline{\Sub_{\findex}(\Gamma)}
\cap\Sub_{\fg}(\Gamma)$.
Take a sequence $H_n$ of finite-index subgroups of $H$
converging to $H$. Fix a finite generating set $T$ for $H$.
There is $n_0\geq 1$ such that for $n\geq n_0$,
$T\subset H_n$, hence $H\subset H_n$.
Therefore, $H=\cap_{n=n_0}^{\infty}H_n$,
and so $H$ is profinitely-closed.
\end{proof}
The group $\Gamma$ is \emph{LERF} (locally extended residually finite)
if every finitely-generated subgroup of $\Gamma$ is profinitely-closed.
Equivalently, $\Gamma$ is LERF if every subgroup of $\Gamma$ is a limit in
$\Sub(\Gamma)$ of finite-index subgroups.
%
%

\section{Benjamini-Schramm convergence, hyperfiniteness, and applications}
\label{sec:borel}
Consider the compact space $\left[0,1\right]^{\NN}$ and the metric
$d_{\prodname}$ on $\left[0,1\right]^{\NN}$ defined by $d_{\prodname}\left(\left(a_{k}\right)_{k=1}^{\infty},\left(b_{k}\right)_{k=1}^{\infty}\right)=\sum_{k=1}^{\infty}2^{-k}\cdot\left|a_{k}-b_{k}\right|$
and generating the product topology of $\left[0,1\right]^{\NN}$.
Fix an enumeration $\left(\left(\rho_{i},W_{i}\right)\right)_{i=1}^{\infty}$
of all pairs $\left(\rho,W\right)\in\mathbb{Z}_{\geq0}\times2^{\FF}$
satisfying $W_{i}\subset B_{\FF}\left(\rho_{i}\right)$ , namely,
$W_{i}$ is a subset of the ball of radius $\rho_{i}$ in $\FF$.
For an $\FF$-probability-space $X$,
define $p_{i}\left(X\right)=\Pr_{x\in X}\left(\Stab_{\FF}\left(x\right)\cap B_{\FF}\left(\rho_{i}\right)=W_{i}\right)$
and $\calL\left(X\right)=\left(p_{i}\left(X\right)\right)_{i=1}^{\infty}\in\left[0,1\right]^{\NN}$.
For a pair $X,Y$ of $\FF$-probability-spaces, define
$d_{\stat}\left(X,Y\right)=d_{\prodname}\left(\calL\left(X\right),\calL\left(Y\right)\right)$.
So, $\left\{p_i\left(X\right)\right\}_{i=1}^{\infty}$ gives the
``local statistics'' of the stabilizers of the action of $\FF$ on $X$.
Note that $d_{\stat}$ defines a pseudometric on the space of (equivalence classes
of) $\FF$-probability-space, which becomes an actual metric when restricted to
finite $\FF$-sets. Convergence under the $d_{\stat}$ metric is called
\emph{Benjamini-Schramm convergence} (more precisely, it is a directed,
edge-labeled version of Benjamini-Schramm convergence).
\begin{lem}
\label{lem:dstat-weakstar}
Let $\left(X_{n}\right)_{n=1}^{\infty}$ and $\left(Y_{n}\right)_{n=1}^{\infty}$
be p.m.p. $\FF$-spaces. Write $\mu_{n}$ and $\nu_{n}$ for
the associated IRSs of $X_{n}$ and $Y_{n}$, respectively.
Then:
\begin{enumerate}[label=(\roman*)]
\item If $\mu_{n}\rightarrow\mu$ for some $\mu\in\IRS(\FF)$
and $d_{\stat}\left(X_{n},Y_{n}\right)\rightarrow0$, then $\nu_{n}\rightarrow\mu$.
\item If $\mu_{n}\rightarrow\lambda$ and $\nu_{n}\rightarrow\lambda$
for some $\lambda\in\IRS(\FF)$, then
$d_{\stat}\left(X_{n},Y_{n}\right)\rightarrow0$
\end{enumerate}
\end{lem}
\begin{proof}
(i) Take $r\geq1$ and $W\subset B_{\FF}\left(r\right)$.
Under the hypothesis of (i),
$\mu_n\left(C_{r,W}\right) \overset{n\rightarrow\infty}{\longrightarrow}
\mu\left(C_{r,W}\right)$ and
$\left|\mu_n\left(C_{r,W}\right)-\nu_n\left(C_{r,W}\right)\right|
\overset{n\rightarrow\infty}{\longrightarrow}0$.
Hence, $\nu_n\left(C_{r,W}\right)
\overset{n\rightarrow\infty}{\longrightarrow}
\mu\left(C_{r,W}\right)$, and so
$\nu_n\rightarrow\mu$. (ii)
Take $r\geq1$ and $W\subset B_{\FF}\left(r\right)$.
Under the hypothesis of (ii), 
$\mu_n\left(C_{r,W}\right) \overset{n\rightarrow\infty}{\longrightarrow}
\lambda\left(C_{r,W}\right)$ and
$\mu_n\left(C_{r,W}\right) \overset{n\rightarrow\infty}{\longrightarrow}
\lambda\left(C_{r,W}\right)$.
Hence, 
$\left|\mu_n\left(C_{r,W}\right)-\nu_n\left(C_{r,W}\right)\right|
\overset{n\rightarrow\infty}{\longrightarrow}0$,
and so $d_{\stat}\left(X_{n},Y_{n}\right)\rightarrow0$.
\end{proof}
We now generalize Definition \ref{def:gen-metric}, and relate $d_{\gen}$
and $d_{\stat}$.
\begin{defn}
\label{def:gen-metric-probspace}Let $\left(X,\mu\right)$ and $\left(Y,\nu\right)$
be p.m.p. $\FF$-spaces. For a measured-space
isomorphism $f\colon X\rightarrow Y$, define
\[
\|f\|_{\gen}=\frac{1}{\left|S\right|}\cdot\sum_{s\in S}\mu\left(\left\{ x\in X\mid f\left(s\cdot x\right)\neq s\cdot f\left(x\right)\right\} \right)
\]
Finally, let 
\[
d_{\gen}\left(X,Y\right)=\inf\left\{ \|f\|_{\gen}\mid f\colon X\rightarrow Y\,\,\text{is a measured-space isomorphism}\right\} \text{.}
\]
\end{defn}

\begin{prop}
\label{prop:dgen-dstat}Let $\left\{ \left(X_{n},\nu_{n}\right)\right\} _{n=1}^{\infty}$
and $\left\{ \left(Y_{n},\lambda_{n}\right)\right\} _{n=1}^{\infty}$
be sequences of p.m.p.
$\FF$-spaces satisfying
$d_{\gen}\left(X_{n},Y_{n}\right)\rightarrow0$.
Then, $d_{\stat}\left(X_{n},Y_{n}\right)\rightarrow0$.
\end{prop}

\begin{proof}
Take a sequence $\left\{ f_{n}\right\} _{n=1}^{\infty}$ of measured-space
isomorphisms $f_{n}\colon X_{n}\rightarrow Y_{n}$ such that $\|f_{n}\|_{\gen}\rightarrow0$.
For $n\in\NN$, let
\[
P_{n}=\cup_{s\in S}\left\{ x\in X_{n}\mid f_{n}\left(s\cdot x\right)\neq s\cdot f_{n}\left(x\right)\right\}\text{ ,}
\]
and for $r\in\NN$, let $Q_{n}^{r}=X_{n}\setminus B_{X_{n}}\left(P_{n},r\right)$
(in the notation of Section \ref{sec:notation}).
Then, $\left\{ Q_{n}^{r}\right\} _{n,r\in\NN}$
are Borel sets, and for each $r\in\NN$, 
\[
\nu\left(B_{X_n}\left(P_n,r\right)\right)\leq
\left(2\cdot \left|S\right|\right)^{r+1}\cdot\nu_n\left(P_n\right)
\overset{n\rightarrow\infty}{\longrightarrow} 0 \text{ ,}
\]
hence
$\nu_{n}\left(Q_{n}^{r}\right)\overset{n\rightarrow\infty}{\longrightarrow}1$.
Furthermore, for each $x\in Q_{n}^{r}$, $\Stab_{\FF}\left(x\right)\cap B_{\FF}\left(r\right)=\Stab_{\FF}\left(f_{n}\left(x\right)\right)\cap B_{\FF}\left(r\right)$.
Hence, for each $i\in\NN$, there is $r\in\NN$ such that $\left|p_{i}\left(X_{n}\right)-p_{i}\left(Y_{n}\right)\right|\leq1-\nu_{n}\left(Q_{n}^{r}\right)$
for all $n\in\NN$, and so $\left|p_{i}\left(X_{n}\right)-p_{i}\left(Y_{n}\right)\right|\overset{n\rightarrow\infty}{\longrightarrow}0$.
Thus, $d_{\stat}\left(X_{n},Y_{n}\right)\rightarrow0$.
\end{proof}
In Proposition \ref{prop:our-elek-newman-sohler} below,
we will give a partial converse to
Proposition \ref{prop:dgen-dstat} in the context of actions of an amenable
group on finite sets.


Let $X$ be a standard Borel space.
Let $E$ be a \emph{Borel equivalence relation} on $X$, i.e., $E\subset X\times X$
is a Borel set which is an equivalence relation.
We write $x\sim_{E}y$ if $\left(x,y\right)\in E$.
Let $E\subset X\times X$ be a Borel equivalence relation.
Then, $E$ is \emph{finite} (resp.\ \emph{countable})
if all of its equivalence classes are finite (resp.\ countable).
A countable equivalence relation $E$ is \emph{hyperfinite}
if it can be written as an \emph{ascending} union of finite Borel equivalence relations.
If $\mu$ is a probability measure on $X$, then $E$ is
\emph{hyperfinite $\mu$-a.e}.
if there is a $\mu$-co-null Borel subset $X_0\subset X$, respecting $E$,
such that the restriction of $E$ to $X_0$ is hyperfinite.
A Borel action $\Gamma\curvearrowright X$ gives rise to a Borel equivalence
relation on $X$ which we denote by $E_{X}^{\Gamma}$.
If $(X,\mu)$ is a p.m.p. $\Gamma$-space, then
the action $\Gamma\curvearrowright \left(X,\mu\right)$ is called \emph{hyperfinite}
if the equivalence relation $E_{X}^{\Gamma}$ is hyperfinite $\mu$-a.e.
A well-known theorem of Ornstein-Weiss \cite{OW} says that every action of an
amenable group is hyperfinite. For a thorough treatment of Borel equivalence
relations, see \cite{KM}.

\begin{defn}
Let $\calX$ be a family of finite graphs. Then, $\calX$ is \emph{hyperfinite}
if for every $\epsilon>0$ there is $K\in\NN$, such that for each
graph $X\in\calX$, there is a set $Z\subset V\left(X\right)$, $\left|Z\right|<\epsilon\cdot\left|V\left(X\right)\right|$,
such that after removing from $X$ all edges incident to $Z$,
each component of the resulting graph is of size at most $K$.
\end{defn}
\begin{prop}
\label{prop:amenable-schreier-hyperfinite}
Assume that $\Gamma$ is amenable. Then, the sequence $\left\{ X_{n}\right\} _{n=1}^{\infty}$
of all finite Schreier graphs of $\Gamma$ is hyperfinite.
\end{prop}

\begin{proof}
For $n\in\NN$ and $Z\subset V\left(X_{n}\right)$, let $c_{n,Z}$
be the size of the largest component of the graph resulting from removing
all edges incident to $Z$ from $X_{n}$. For $\epsilon>0$, let $c_{n,\epsilon}=\min\left\{ c_{n,Z}\mid Z\subset V\left(X_{n}\right),\left|Z\right|<\epsilon\cdot\left|V\left(X_{n}\right)\right|\right\} $.
Assume, for the sake of contradiction, that $\left\{ X_{n}\right\} _{n=1}^{\infty}$
is not a hyperfinite family.
Then, there is $\epsilon>0$ and an increasing
sequence $\left\{ n_{k}\right\} _{n=1}^{\infty}$ such that
$c_{n_{k},\epsilon}\overset{k\rightarrow\infty}{\longrightarrow}\infty$.
Write $\mu_n\in\IRS(\Gamma)$ for the IRS associated with $X_n$.
Since $\IRS(\Gamma)$ is compact, we may further assume
that $\mu_n\overset{n\rightarrow\infty}{\longrightarrow}\mu$
for some $\mu\in\IRS(\Gamma)$.
By Proposition 13 of \cite{agv},
there is a p.m.p. $\Gamma$-space $X$
whose associated IRS is $\mu$.
By \cite{OW}, the action $\Gamma\curvearrowright X$ is hyperfinite since
$\Gamma$ is amenable. Then, by Theorem 1.1 of \cite{schramm2008} (see also Theorem 1 of \cite{elek2012}),
$\left\{ X_{n_{k}}\right\} _{k=1}^{\infty}$
is a hyperfinite family, a contradiction.
\end{proof}

Recall that a bijection $f\colon X\rightarrow Y$ between measured-spaces
$\left(X,\mu\right)$ and $\left(Y,\nu\right)$ is a \emph{measured-space isomorphism} if $f$ and $f^{-1}$ are both Borel maps, and for each
Borel set $A\subset X$, $\mu\left(A\right)=\nu\left(f\left(A\right)\right)$.
A measured-space isomorphism from a measured-space 
to itself is called a
\emph{measured-space automorphism}.

\begin{prop}
\label{prop:amenable-co-sofic}
Assume that $\Gamma$ is amenable and $\mu\in\IRS(\Gamma)$. Then,
$\mu$ is co-sofic in $\FF$.
\end{prop}

\begin{proof}
By Proposition 13 of \cite{agv}, there is a p.m.p.
$\Gamma$-space $\left(X,\nu\right)$
for which $\mu$ is the associated IRS. Since $\Gamma$ is amenable, the action
$\Gamma\curvearrowright X$ is hyperfinite by \cite{OW}.
From now on, regard $X$ as an $\FF$-probability-space,
and for each $s\in S$ (recall that $S$ is our fixed basis for $\FF$),
let $f_s\colon X\rightarrow X$ be the measured-space automorphism defined by
$f_s\left(x\right)=s\cdot x$.

Let $n\geq1$. Then, there is a Borel set $Z\subset X$,
$\nu\left(Z\right)<\frac{1}{n}$, such that all orbits of the restriction of
the equivalence relation $E_{X}^{\FF}$ to $X\setminus Z$ are finite.
Let $E=E_{X}^{\FF}\mid_{X\setminus Z} \cup
\left\{\left(x,x\right)\mid x\in Z\right\}$.
For each $s\in S$, Lemma $\ref{lem:pmp-bijection}$ below gives us a
measured-space automorphism
$h_s\colon X\rightarrow X$ which respects the equivalence relation
$E$ and agrees with $f_s$
on $X\setminus\left(Z\cup\left(f_s\right)^{-1}\left(Z\right)\right)$.
Let $X_n$ be the p.m.p. $\FF$-space which, as a probability space, equals $X$, endowed with the p.m.p. action of $\FF$ given by $s\cdot x=h_s\left(x\right)$
for each $s\in S$.
Then, $d_{\gen}\left(X_{n},X\right)\rightarrow0$,
and so, by Proposition \ref{prop:dgen-dstat}, $d_{\stat}\left(X_{n},X\right)\rightarrow0$.
Write $\mu_{n}\in\IRS(\FF)$ for
the IRS associated with $X_n$.
Then $\mu_n$ is a finite-index IRS since each $h_s$ respects
the finite Borel equivalence relation $E$.
By Lemma \ref{lem:dstat-weakstar}(i),
applied to the sequence $\left(X_n\right)_{n=1}^{\infty}$ against the constant
sequence $\left(X\right)_{n=1}^{\infty}$,
$\mu_{n}\rightarrow\mu$,
and so, $\mu$ is co-sofic in $\FF$.
\end{proof}

\begin{lem}
\label{lem:pmp-bijection}
Let $X$ be a probability space, $f\colon X\rightarrow X$ a measured-space
automorphism and $E\subset X\times X$ a finite Borel equivalence
relation on $X$.
Write $X_{E,f}=\{x\in X\mid f\left(x\right)\sim_{E} x\}$.
Then, there is a measured-space automorphism $h\colon X\rightarrow X$
such that $h\left(x\right)\sim_{E}x$ for every $x\in X$,
and $h$ agrees with $f$ on $X_{E,f}$.
\end{lem}
\begin{proof}
The idea behind the construction of the map $h$ is as follows:
Since the equivalence relation $E$ is finite,
the space $X$ can be decomposed as a disjoint union of
``finite $f$-cycles'' and
``finite maximal $f$-chains''.
That is, sets of the form
$x,f(x),f^{(2)}(x),\dotsc,f^{(m)}(x)$ for $x\in X$ and $m\geq 0$, such that
$f^{(i)}(x)\sim_{E}f^{(i+1)}(x)$ for each $0\leq i<m$,
and such that either $f^{(m+1)}(x)=x$ and $f^{(m)}(x)\sim_{E}x$
(these are the $f$-cycles), or
$f^{-1}(x)\nsim_{E}x$ and $f^{(m)}(x)\nsim_{E}f^{(m+1)}(x)$
(these are the maximal $f$-chains).
For each maximal $f$-chain, as above,
we define $h(f^{(i)}(x))=f^{(i+1)}(x)$ for each $0\leq i<m$
and $h(f^{(m)}(x))=x$.
On the $f$-cycles, we make $h$ identical to $f$.
The resulting function $h$ is a bijection.
We now formalize this construction in a way that enables us to see that the resulting
map $h$ is a Borel measure-preserving automorphism of $X$.

For each $n\geq 0$, let $A_n=\left\{x\in X\setminus X_{E,f}\mid
\left(f^{-1}\right)^{\left(n\right)}\left(x\right)\sim_{E}x\right\}$.
For each $n\geq 0$, set $X_n=\cap_{i=0}^{n}A_i\setminus A_{n+1}$.
So, a point $x\in X\setminus X_{E,f}$ satisfies $x\in X_n$ if and only if
$\left(f^{-1}\right)^{\left(k\right)}\left(x\right)
\sim_{E}x$ for every $1\leq k\leq n$, but
$\left(f^{-1}\right)^{\left(n+1\right)}\left(x\right)\nsim_{E}x$.
The sets $X_n$ are disjoint by construction.
Furthermore, since each equivalence class of $E$ is finite,
every $x\in X\setminus X_{E,f}$ belongs to $X_n$ for some $n\in\NN$.
Therefore, $\calC=\left\{X_{E,f}\right\}\cup\left\{X_n\right\}_{n=0}^{\infty}$ forms
a partition of $X$ into countably many Borel sets.
We define $h\colon X\rightarrow X$:
for $x\in X_{E,f}$, set $h\left(x\right)=f\left(x\right)$,
and for $x\in X_n$, set $h\left(x\right)=
\left(f^{-1}\right)^{\left(n\right)}\left(x\right)$.
Then, $h$ is a bijection.
By the definition of $h$ and since $f$ is a Borel automorphism, $h$ maps every
Borel subset of each set in the partition $\calC$ to a Borel subset of $X$.
Thus, $h$ maps every Borel subset of $X$ to a Borel set.
This shows that $h^{-1}$ is a Borel bijection, and so, since $X$ is a standard Borel
space, $h$ is a Borel bijection as well
(see Corollary 15.2 in \cite{kechris}).
Similarly, $h$ preserves the measure on $X$ because it does so when restricted
to each set in the partition $\calC$.
\end{proof}

The following result, which is essential for our needs, gives a converse to
Proposition \ref{prop:dgen-dstat} in case the $\FF$-sets are finite and the actions
in one of the sequences factor through an amenable quotient.
\begin{prop}
\label{prop:our-elek-newman-sohler}Assume that $\Gamma$ is amenable. Let
$\left(X_{n}\right)_{n=1}^{\infty}$ be a sequence of finite $\FF$-sets
and $\left(Y_{n}\right)_{n=1}^{\infty}$ a sequence of finite
$\Gamma$-sets,
satisfying $\left|X_{n}\right|=\left|Y_{n}\right|$ and
$d_{\stat}\left(X_{n},Y_{n}\right)\rightarrow0$. Then, $d_{\gen}\left(X_{n},Y_{n}\right)\rightarrow0$.
\end{prop}

\begin{proof}
The statement of this proposition is an adaptation of a theorem of
Newman and Sohler (see \cite{newman-sohler2011} and \cite{newman-sohler2013})
from the context of finite undirected
graphs, to the context of group actions on finite sets. We begin by
describing the Newman-Sohler Theorem (see Theorem 5 of \cite{elek2012}
for this formulation, and a different proof). First, we need some
definitions. Fix $q\in\NN$. Let ${\bf P}_{q}$ be the collection
of finite undirected graphs for which the degree of each vertex is
at most $q$. We begin by defining the notion of the statistical distance
between finite undirected graphs. For $H\in{\bf P}_{q}$ and a vertex $h_{0}$
of $H$, we say that $\left(H,h_{0}\right)$ is a pointed graph of
radius $r$ if each vertex $h\in H$ is at distance at most $r$ from
$h_{0}$. Write ${\bf P}_{q,r}$ for the set of pointed graphs $\left(H,h_{0}\right)$
of radius $r$ with $H\in{\bf P}_{q}$. Enumerate the disjoint union
$\coprod_{r\in\NN}{\bf P}_{q,r}$ by $\left\{ H_{i}\right\} _{i=1}^{\infty}$.
For $r\in\NN$, $H_{i}\in{\bf P}_{q,r}$ and $G\in{\bf P}_{q}$, write
$p_{i}\left(G\right)$ for the probability, under a uniformly random
choice of a vertex $v$ of $G$, that the ball of radius $r$, centered
at $v$, is pointed-isomorphic to $H_{i}$. Write $\calL\left(G\right)=\left(p_{i}\left(G\right)\right)_{i=1}^{\infty}\in\left[0,1\right]^{\NN}$.
For $G_{1},G_{2}\in{\bf P}_{q}$, the statistical distance between
$G_{1}$ and $G_{2}$ is defined as $d_{\stat}\left(G_{1},G_{2}\right)=\sum_{i=1}^{\infty}2^{-i}\cdot\left|p_{i}\left(G_{1}\right)-p_{i}\left(G_{2}\right)\right|$.
We now define another notion of distance between graphs in ${\bf P}_{q}$
(the generator-metric $d_{\gen}$ is its analogue in the context of
group actions). For $G_{1},G_{2}\in{\bf P}_{q}$, $n:=\left|V(G_{1}\right)|=\left|V(G_{2})\right|$, and a bijection $f:V\left(G_{1}\right)\rightarrow V\left(G_{2}\right)$, let $Q_f$ be the set of pairs
$\left(v_{1},v_{2}\right)$ of vertices of $G_{1}$, such that $\left(v_{1},v_{2}\right)$
is an edge of $G_{1}$, but $\left(f\left(v_{1}\right),f\left(v_{2}\right)\right)$
is not an edge of $G_{2}$, or vice versa. Let
$\|f\|=\frac{1}{n}\cdot|Q_f|$.
Define $d\left(G_{1},G_{2}\right)$ as the minimum of $\|f\|$, running over all
bijections $\|f\|$ between the vertex sets.

The Newman-Sohler Theorem says that if ${\bf G}\subset{\bf P}_q$ is a hyperfinite
family, then for every $\epsilon>0$, there is $f\left(\epsilon\right)>0$,
such that for every $G_{1}\in{\bf P}_{q}$ and $G_{2}\in{\bf G}$,
if $\left|V\left(G_{1}\right)\right|=\left|V\left(G_{2}\right)\right|$
and $d_{\stat}\left(G_{1},G_{2}\right)<f\left(\epsilon\right)$, then
$d\left(G_{1},G_{2}\right)<\epsilon$. In other words, if $\left(G_{n}^{(1)}\right)_{n=1}^{\infty}$
is a sequence in ${\bf P}_{q}$ and $\left(G_{n}^{(2)}\right)_{n=1}^{\infty}$
is a sequence in the hyperfinite family ${\bf G}$, satisfying $\left|G_{n}^{(1)}\right|=\left|G_{n}^{(2)}\right|$,
then $d_{\stat}\left(G_{n}^{(1)},G_{n}^{(2)}\right)\rightarrow0$ implies
$d\left(G_{n}^{(1)},G_{n}^{(2)}\right)\rightarrow0$.

To adapt the Newman-Sohler Theorem to the context of group actions,
we use a standard encoding of actions of $\FF$ by undirected graphs.
The details of this type of encoding are described, for example, in the
proof of Theorem 9 of \cite{elek2012}: There is $q\in\NN$ and a mapping
$U$ from the set of actions of $\FF$ on finite sets to the set
${\bf P}_{q}$ with the
following properties: (1) If $\left(X_{n}\right)_{n=1}^{\infty}$
and $\left(Y_{n}\right)_{n=1}^{\infty}$ are sequences of finite $\FF$-sets,
then $d_{\stat}\left(X_{n},Y_{n}\right)\rightarrow0$ if and only
if $d_{\stat}\left(U\left(X_{n}\right),U\left(Y_{n}\right)\right)\rightarrow0$,
and (2) if, in addition, $\left|X_{n}\right|=\left|Y_{n}\right|$, then $d_{\gen}\left(X_{n},Y_{n}\right)\rightarrow0$
if and only $d\left(U\left(X_{n}\right),U\left(Y_{n}\right)\right)\rightarrow0$.

The proposition follows at once from the Newman-Sohler Theorem and
the above properties of the encoding function $U$.
\end{proof}

\begin{rem}
The assumption that $\Gamma$ is amenable in Proposition
\ref{prop:our-elek-newman-sohler} is essential.
Indeed, for $d\geq 2$, take $\Gamma=\FF_d$, the free group on $d$ generators.
Then, there are sequences $(\Lambda_n)_{n=1}^{\infty}$ and
$(\Delta_n)_{n=1}^{\infty}$ of finite quotients of $\Gamma$,
$|\Lambda_n|=2\cdot|\Delta_n|$,
giving rise to $2d$-regular Cayley graphs $X_n=\Cay\left(\Lambda_n\right)$
and $Y_n=\Cay\left(\Delta_n\right)$,
such that $(X_n)_{n=1}^{\infty}$
is a family of expander graphs,
and such that the girths of both $X_n$ and $Y_n$ approach
infinity as $n\rightarrow\infty$
(see Theorem 7.3.12 of \cite{lubotzky-expanders}
for examples of families of expander graphs with large girth).
Then, $d_{\stat}\left(X_n,Y_n\coprod Y_n\right)\rightarrow 0$ since for every
radius $r\geq 1$, any ball of radius $r$ in $X_n$ and in $Y_n$ is a tree for large enough $n$. But, since $(Y_n\coprod Y_n)_{n=1}^{\infty}$ is a highly
non-expanding family, $d_{\gen}(X_n,Y_n)$ does not approach $0$ as
$n\rightarrow\infty$
\end{rem}

\section{The main theorem}
\label{sec:main-theorems}
\begin{defn}
A sequence $\left(X_{n}\right)_{n=1}^{\infty}$ of finite $\FF$-sets
is \emph{convergent} if it has a limiting IRS,
i.e. the sequence $\left(\mu_n\right)_{n=1}^{\infty}$ of IRSs associated with
$\left(X_{n}\right)_{n=1}^{\infty}$ converges in $\IRS(\FF)$.
\end{defn}
\begin{defn}
A sequence $\left(X_{n}\right)_{n=1}^{\infty}$ of finite $\FF$-sets
is a \emph{stability-challenge for $\Gamma$} if $\Pr_{x\in X_{n}}\left(w\cdot x=x\right)\rightarrow1$
for each $w\in\Ker\pi$.
\end{defn}
Note that $\IRS(\FF)$ is compact, and so every stability-challenge for $\Gamma$ has
a convergent subsequence.
\begin{defn}
Let $\left(X_{n}\right)_{n=1}^{\infty}$ be a stability-challenge for $\Gamma$.
Then, a sequence $\left(Y_{n}\right)_{n=1}^{\infty}$ of finite $\Gamma$-sets,
satisfying $\left|X_n\right|=\left|Y_n\right|$, is:
\begin{enumerate}[label=(\roman*)]
\item
a \emph{solution}
for $\left(X_{n}\right)_{n=1}^{\infty}$, if
$d_{\gen}\left(X_n,Y_n\right)\overset{n\rightarrow\infty}{\longrightarrow}0$.
\item
a \emph{statistical-solution}
for $\left(X_{n}\right)_{n=1}^{\infty}$, if
$d_{\stat}\left(X_n,Y_n\right)\overset{n\rightarrow\infty}{\longrightarrow}0$.
\end{enumerate}
\end{defn}
By Proposition \ref{prop:dgen-dstat}, if $\left(X_{n}\right)_{n=1}^{\infty}$
is a stability-challenge for $\Gamma$, then any solution for
$\left(X_{n}\right)_{n=1}^{\infty}$ is a statistical-solution.
By Proposition \ref{prop:our-elek-newman-sohler}, the converse holds as well if $\Gamma$ is amenable.

Note that $\Gamma$ is P-stable (Definition \ref{def:stable-group})
if and only if every stability-challenge for $\Gamma$
has a solution. In fact, it suffices to consider \emph{convergent}
stability-challenges:
\begin{lem}
\label{lem:stability-by-convergent-solutions}
The group $\Gamma$ is P-stable if and only if every convergent stability-challenge
for $\Gamma$ has a solution.
\end{lem}
\begin{proof}
We only need to prove the ``if'' direction.
Assume that $\Gamma$ is not P-stable.
We would like to show that $\Gamma$ has a convergent stability-challenge
which does not have a solution.
Take $\epsilon>0$ and a sequence $\left(X_{n}\right)_{n=1}^{\infty}$
such that $X_{n}$ is a $\left(\delta_n,E_n\right)$-almost-$\Gamma$-set
for $\delta_n=\frac{1}{n}$ and $E_n=\Ker\pi\cap B_{\FF}\left(n\right)$,
but there is no $\Gamma$-set $Y_{n}$
for which $d_{\gen}\left(X_{n},Y_{n}\right)<\epsilon$.
Then, every subsequence of $\left(X_{n}\right)_{n=1}^{\infty}$
is a stability-challenge
for $\Gamma$ which has no solution.
Since $\IRS(\FF)$ is compact,
$\left(X_{n}\right)_{n=1}^{\infty}$ has a subsequence which is a convergent
stability-challenge for $\Gamma$ with no solution.
\end{proof}
\begin{lem}
\label{lem:challenge-vs-limit-in-gamma}
Let $\left(X_{n}\right)_{n=1}^{\infty}$ be a convergent
sequence of finite $\FF$-sets,
and write $\mu\in\IRS(\FF)$ for its limiting IRS.
Then, $\left(X_{n}\right)_{n=1}^{\infty}$ is a stability-challenge for $\Gamma$
if and only if $\mu\in\IRS(\Gamma)$.
\end{lem}

\begin{proof}
Let $\left(\mu_{n}\right)_{n=1}^{\infty}$ be the sequence of IRSs
associated with $\left(X_{n}\right)_{n=1}^{\infty}$.
By definition, $\mu_n\rightarrow\mu$.
Then, $\left(\mu_{n}\right)_{n=1}^{\infty}$ is a stability-challenge for $\Gamma$
if and only if $\mu_n\left(C_w\right)\rightarrow 1$ for each $w\in\Ker\pi$.
If $\mu\in\IRS(\Gamma)$, then $\mu_n\left(C_w\right)\rightarrow\mu\left(C_w\right)=1$ for every $w\in\Ker\pi$,
which says that $\left(\mu_{n}\right)_{n=1}^{\infty}$
is a stability-challenge for $\Gamma$.
In the other direction, if $\left(\mu_{n}\right)_{n=1}^{\infty}$ is a
stability-challenge for $\Gamma$, then for every $w\in\Ker\pi$,
$\mu_{n}\left(C_{w}\right)\rightarrow1$,
but $\mu_{n}\left(C_{w}\right)\rightarrow\mu\left(C_{w}\right)$,
forcing $\mu\left(C_{w}\right)=1$. Therefore,
\[
\mu\left(\left\{ H\leq\FF\mid\Ker\pi\leq H\right\} \right)=\mu\left(\cap_{w\in\Ker\pi}C_{w}\right)=1\text{,}
\]
and so $\mu\in\IRS(\Gamma)$.
\end{proof}
The proof of the following lemma is an adaptation of the argument
in the proof of Proposition 6.1 of \cite{arzhantseva-paunescu}.
\begin{lem}
\label{lem:amplification-sparsification}Let $\left(X_{n}\right)_{n=1}^{\infty}$
be a sequence of finite $\Gamma$-sets satisfying
$\left|X_{n}\right|\rightarrow\infty$,
with associated sequence of IRSs
$\left(\mu_{n}\right)_{n=1}^{\infty}$,
satisfying $\mu_{n}\rightarrow\mu$
for some $\mu\in\IRS(\Gamma)$.
Let $\left(m_{k}\right)_{k=1}^{\infty}$ be a sequence of positive
integers satisfying $m_{k}\rightarrow\infty$.
Then, there is a sequence
$\left(Y_{k}\right)_{k=1}^{\infty}$
of $\Gamma$-sets, satisfying $\left|Y_{k}\right|=m_{k}$, with associated
IRS sequence $\left(\nu_{k}\right)_{k=1}^{\infty}$,
such that $\nu_{k}\rightarrow\mu$ as well.
\end{lem}
\begin{proof}
For an integer $r\geq0$, write $Z_{r}$ for the $\Gamma$-set on $r$
points on which $\Gamma$ acts trivially. Take an increasing sequence $\left(i_{n}\right)_{n=1}^{\infty}$
of positive integers such that for every $n\in\NN$ and $i_{n}\leq k<i_{n+1}$,
$\frac{\left|X_{n}\right|}{m_{k}}<\frac{1}{n}$. For
each $k\geq i_1$, take the unique $n\in\NN$ for which $i_{n}\leq k<i_{n+1}$,
write $m_{k}=q_{k}\cdot\left|X_{n}\right|+r_{k}$ for integers $q_{k}\geq n$
and $0\leq r_{k}<\left|X_{n}\right|,$ and let $Y_{k}=\left(X_{n}\right)^{\coprod q_{k}}\coprod Z_{r_{k}}$.
So, $\left|Y_{k}\right|=m_{k}$.
For $1\leq k<i_{1}$, define $Y_{k}=Z_{m_k}$.
Let $\left(\nu_{k}\right)_{k=1}^{\infty}$
be the sequence of IRSs associated with $\left(Y_{k}\right)_{k=1}^{\infty}$.
We would like to show that $\nu_{k}\rightarrow\mu$. Take a continuous
function $f\colon\Sub(\Gamma)\rightarrow\RR$. Let $n\in\NN$
and $i_{n}\leq k<i_{n+1}$. Then,
\begin{align*}
\left|\int fd\nu_{k}-\int fd\mu_{n}\right|= & \left|\frac{1}{m_{k}}\cdot\left(q_{k}\cdot\left|X_{n}\right|\cdot\int fd\mu_{n}+r_{k}\cdot f(\Gamma)\right)-\int fd\mu_{n}\right|\\
= & \frac{1}{m_{k}}\cdot\left|\left(q_{k}\cdot\left|X_{n}\right|-m_{k}\right)\cdot\int fd\mu_{n}+r_{k}\cdot f(\Gamma)\right|\\
= & \frac{r_{k}}{m_{k}}\cdot\left|-\int fd\mu_{n}+f(\Gamma)\right|\\
< & \frac{2}{n}\cdot\|f\|_{\infty}
\end{align*}
Therefore, $\int fd\nu_{k}\rightarrow\int fd\mu$, and so $v_{k}\rightarrow\mu$.
\end{proof}

\begin{defn}
\label{def:co-sofic-vs-challenge}
Let $\left(X_{n}\right)_{n=1}^{\infty}$ be a convergent stability-challenge for
$\Gamma$ whose limiting IRS is $\mu\in\IRS(\Gamma)$.
Then, $\left(X_{n}\right)_{n=1}^{\infty}$ is \emph{co-sofic} if $\mu$ is
co-sofic.
\end{defn}

\begin{lem}
\label{lem:co-sofic-vs-stat-solvable}
Let $\left(X_{n}\right)_{n=1}^{\infty}$ be a convergent stability-challenge for
$\Gamma$.
Then, $\left(X_{n}\right)_{n=1}^{\infty}$ is co-sofic if and only if
$\left(X_{n}\right)_{n=1}^{\infty}$ has a statistical-solution.
\end{lem}
\begin{proof}
Write $\mu\in\IRS(\Gamma)$
for the limiting IRS of $\left(X_{n}\right)_{n=1}^{\infty}$.
Assume that $\left(X_{n}\right)_{n=1}^{\infty}$ is co-sofic.
By Lemma \ref{lem:approximation-of-cosofic-by-actions}, there is a sequence
$\left(Y_{n}\right)_{n=1}^{\infty}$ of finite $\Gamma$-sets whose associated
sequence of IRSs converges to $\mu$.
By Lemma \ref{lem:amplification-sparsification}, we may assume that
$\left|X_n\right|=\left|Y_n\right|$ for all $n\in\NN$.
Then, by Lemma \ref{lem:dstat-weakstar}(ii), $\left(Y_{n}\right)_{n=1}^{\infty}$
is a statistical-solution for $\left(X_{n}\right)_{n=1}^{\infty}$.

In the other direction, assume that $\left(X_{n}\right)_{n=1}^{\infty}$ has a
statistical-solution $\left(Y_{n}\right)_{n=1}^{\infty}$.
Write $\left(\nu_n\right)_{n=1}^{\infty}$ for the sequence of IRSs associated
with $\left(Y_{n}\right)_{n=1}^{\infty}$.
Then, $\left(\nu_n\right)_{n=1}^{\infty}$ is a sequence of finite-index IRSs
in $\IRS(\Gamma)$, which, by Lemma \ref{lem:dstat-weakstar}(i), converges
to $\mu$, and so $\left(X_{n}\right)_{n=1}^{\infty}$ is co-sofic.
\end{proof}

\begin{lem}
\label{lem:F-co-sofic-gives-challenge}
Let $\mu\in\IRS(\Gamma)$ and assume that $\mu$ is co-sofic in $\FF$.
Then, there is a convergent stability-challenge for $\Gamma$
whose limiting IRS is $\mu$.
\end{lem}
\begin{proof}
By Lemma \ref{lem:approximation-of-cosofic-by-actions} applied to $\FF$
(rather than $\Gamma$), there is a sequence
$\left(X_{n}\right)_{n=1}^{\infty}$ of finite $\FF$-sets
whose associated sequence of IRSs converges to $\mu$.
But $\mu\in\IRS(\Gamma)$, and so, by Lemma
\ref{lem:challenge-vs-limit-in-gamma}, $\left(X_{n}\right)_{n=1}^{\infty}$ is a
stability-challenge for $\Gamma$.
\end{proof}

The following theorem proves Theorem \ref{thm:intro-main} of the introduction.
\begin{thm}
\label{thm:main-theorem}
~
\begin{enumerate}[label=(\roman*)]
\item
Assume that $\Gamma$ is P-stable
and $\mu\in\IRS(\Gamma)$ is co-sofic in $\FF$.
Then, $\mu$ is co-sofic in $\Gamma$.
\item
Assume that $\Gamma$ is amenable.
Then, $\Gamma$ is P-stable if and only if every $\mu\in\IRS(\Gamma)$
is co-sofic in $\Gamma$.
\end{enumerate}
\end{thm}
\begin{proof}
(i) By Lemma \ref{lem:F-co-sofic-gives-challenge}, there is a convergent
stability-challenge
$\left(X_{n}\right)_{n=1}^{\infty}$ for $\Gamma$ whose limiting IRS is $\mu$.
Then, $\left(X_{n}\right)_{n=1}^{\infty}$ has a solution, a fortiori, it has a
statistical solution. Thus, by Lemma \ref{lem:co-sofic-vs-stat-solvable},
$\left(X_{n}\right)_{n=1}^{\infty}$ is co-sofic, i.e., $\mu$ is co-sofic in
$\Gamma$.

(ii) Assume that $\Gamma$ is P-stable. Let $\mu\in\IRS(\Gamma)$.
By Proposition \ref{prop:amenable-co-sofic}, $\mu$ is co-sofic in $\FF$.
Hence, by (i), $\mu$ is co-sofic in $\Gamma$.

In the other direction, assume that every
$\mu\in\IRS(\Gamma)$ is co-sofic in $\Gamma$.
Let $\left(X_{n}\right)_{n=1}^{\infty}$ be a convergent stability-challenge for
$\Gamma$.
Then, $\left(X_{n}\right)_{n=1}^{\infty}$ is co-sofic, and so by Lemma
\ref{lem:co-sofic-vs-stat-solvable},
it has a statistical-solution $\left(Y_{n}\right)_{n=1}^{\infty}$.
By Proposition \ref{prop:our-elek-newman-sohler},
$\left(Y_{n}\right)_{n=1}^{\infty}$ is, in fact, 
a solution for $\left(X_{n}\right)_{n=1}^{\infty}$,
and so $\Gamma$ is P-stable by
Lemma \ref{lem:stability-by-convergent-solutions}.
\end{proof}

\section{Applications of the main theorems}
\label{sec:applications}
In this section, we give several applications of the results of Section
\ref{sec:main-theorems}, and in particular prove Theorem
\ref{thm:intro-examples}.

The next proposition proves Theorem \ref{thm:application-positive}.

\begin{prop}
\label{prop:positive-application}
Assume that $\Sub(\Gamma)$ is countable,
and that every almost-normal subgroup of $\Gamma$ is profinitely-closed.
Then, every $\mu\in\IRS(\Gamma)$ is co-sofic. If, further,
$\Gamma$ is amenable, then $\Gamma$ is P-stable.
\end{prop}

\begin{proof}
The latter statement follows from the former by Theorem
\ref{thm:main-theorem}(ii).
We turn to proving the former.
Since $\Sub(\Gamma)$ is countable, 
every IRS in $\IRS(\Gamma)$ is atomic,
and so all of its atoms are almost-normal subgroups.
Let $\mu\in\IRS(\Gamma)$.
Take a sequence $\left(\mu_n\right)_{n=1}^{\infty}$ in
$\IRS(\Gamma)$ of finitely-supported atomic IRSs, converging to $\mu$.
Since $\IRS(\Gamma)$ is metrizable, it suffices to prove that each
$\mu_n$ is co-sofic. Let $H$ be an almost-normal subgroup of $\Gamma$,
and let $\nu\in\IRS(\Gamma)$ be the atomic IRS assigning measure
$\frac{1}{\left|H^{\Gamma}\right|}$ to each conjugate of $H$.
It suffices to prove that $\nu$ is co-sofic.
Take representatives $g_1,\dotsc,g_k$ for the left cosets of
$N_{\Gamma}\left(H\right)$ in $\Gamma$.
Since $H$ is profinitely-closed in $\Gamma$ and
$\left[\Gamma:N_{\Gamma}\left(H\right)\right]<\infty$,
$H$ is profinitely-closed and normal in $N_{\Gamma}\left(H\right)$.
Therefore, there is a sequence
$\left(H_n\right)_{n=1}^{\infty}$ of finite-index normal subgroups of
$N_{\Gamma}\left(H\right)$ such that $H=\cap_{n=1}^{\infty}H_n$.
For $1\leq i\leq k$, ${^{g_i}H_n}\overset{n\rightarrow\infty}{\longrightarrow}{^{g_i}H}$, and so
$\delta_{^{g_i}H_n}\overset{n\rightarrow\infty}{\longrightarrow}\delta_{^{g_i}H}$.
Hence, writing $\nu_n=\frac{1}{k}\sum_{i=1}^k\delta_{^{g_i}H_n}$,
we have $\nu_n\overset{n\rightarrow\infty}{\longrightarrow}\nu$,
i.e., $\nu$ is a limit of finite-index random-subgroups.
It remains to show that each random-subgroup $\nu_n$ is an IRS.
Take $g\in\Gamma$. Let $\sigma\in\Sym\left(k\right)$ be the permutation for which
$gg_i N_{\Gamma}\left(H\right)=g_{\sigma(i)}N_{\Gamma}\left(H\right)$.
Since $H_n$ is normal in $N_{\Gamma}\left(H\right)$,
${^{gg_i}H_n}={^{g_{\sigma(i)}}H_n}$ for each $1\leq i\leq k$.
So, $g\cdot\nu_n=\frac{1}{k}\sum_{i=1}^k\delta_{^{gg_i}H_n}=\nu_n$,
i.e. $\nu_n$ is an IRS. Hence, $\nu$ is co-sofic.
\end{proof}

The following corollary provides a proof for part (i) of Theorem 
\ref{thm:intro-examples}.
\begin{cor}
\label{cor:polycyclic}
Virtually polycyclic groups are P-stable.
\end{cor}
\begin{proof}
Assume that $\Gamma$ is a virtually polycyclic group.
Then, every subgroup of $\Gamma$ is finitely-generated, and so
$\Sub(\Gamma)$ is countable.
Furthermore, $\Gamma$ is LERF (see \cite{malcev}) and amenable.
Hence, all of the conditions of Proposition \ref{prop:positive-application} are met.
\end{proof}
\begin{rem}
Nevertheless, not every solvable group is P-stable, even if it is residually-finite,
see Corollary \ref{cor:abels}.
\end{rem}

The following corollary provides a proof for part (ii) of Theorem \ref{thm:intro-examples}.
\begin{cor}
\label{cor:bs1n}
For every $n\in \ZZ$, the Baumslag-Solitar group $\BS\left(1,n\right)$ is P-stable.
\end{cor}
\begin{proof}
Let $\Gamma=\BS\left(1,n\right)$.
Note that $\Gamma\cong\ZZ\left[\frac{1}{n}\right]\rtimes\ZZ$ where $1\in\ZZ$ acts on
$\ZZ\left[\frac{1}{n}\right]$ by multiplication by $n$.
We use Proposition \ref{prop:positive-application} to show that $\Gamma$ is P-stable.
First, $\Gamma$ is amenable since it is solvable.
The group $\Gamma$ is an example of a constructible solvable group.
Every constructible solvable group is residually-finite, and
the class of constructible solvable groups is closed under
taking quotients and finite-index subgroups
(see \cite{baumslag-bieri} or Section 11.2 of \cite{lennox-robinson}).
Therefore, every almost-normal subgroup of $\Gamma$ is
profinitely-closed.
It remains to show that $\Sub(\Gamma)$ is countable.
In general, for a countable group $G$ and $N\lhd G$,
if $G/N$ is Noetherian (i.e. every subgroup is finitely-generated)
and $\Sub\left(N\right)$ is countable,
then $\Sub\left(G\right)$ is countable.
In our case, taking $N=\ZZ\left[\frac{1}{n}\right]$, $\Gamma/N$ is infinite cyclic,
and $\Sub\left(N\right)$ is countable.
To see that $N=\ZZ\left[\frac{1}{n}\right]$ indeed has only
countably many subgroups,
we argue as follows. Let $H$ be a subgroup of $\ZZ\left[\frac{1}{n}\right]$.
Then, $H$ is determined by the sequence
$\left(H_i\right)_{i=0}^{\infty}$, where
$H_i=H\cap\frac{1}{n^i}\ZZ$.
The latter is determined by the sequence
$\left(l_i\right)_{i=0}^{\infty}$, where
$l_i=\left[\frac{1}{n^i}\ZZ:H\cap\frac{1}{n^i}\ZZ\right]$.
Assuming, $H\neq\left\{1\right\}$, each $l_i$ is a positive integer
(i.e. $l_i\neq\infty$).
For each $i\geq 0$,
$l_i=l_{i+1}/\gcd\left(n,l_{i+1}\right)$.
This shows that the sequence of sets of prime divisors of the
elements of $\left(l_i\right)_{i=0}^{\infty}$
eventually stabilizes at some finite set of primes
$\left\{p_1,\dotsc,p_m\right\}$.
Let $n_0\geq 0$ be the minimal non-negative integer for which $l_{n_0}$ is
divisible by $p_1\cdots p_m$. Write $q=\gcd\left(n,l_{i+1}\right)$.
Then, $l_i=l_{n_0}\cdot q^{i-n_0}$ for each $i\geq n_0$.
The sequence $\left(l_i\right)_{i=0}^{\infty}$ is determined
by $n_0$, $l_{n_0}$ and $q$. Subsequently,
$\ZZ\left[\frac{1}{n}\right]$ has only countably
many subgroups.
\end{proof}


\begin{prop}
\label{prop:stable-is-fganpc}Assume that $\Gamma$ is
P-stable. Let $H$ be an almost-normal subgroup
of $\Gamma$, such that the IRS $\mu\in\IRS(\Gamma)$, assigning
probability $\frac{1}{\left|H^\Gamma\right|}$ to each conjugate of $H$, is co-sofic in $\FF$. Then, $H$ is a limit in $\Sub\left(\Gamma\right)$
of finite-index subgroups. If, in addition, $H$ is finitely-generated, then $H$ is
profinitely-closed.
\end{prop}
\begin{proof}
The latter statement follows from the former by Lemma \ref{lem:profinite-vs-chabauty}.
We turn to proving the former.
By Theorem \ref{thm:main-theorem}(i), $\mu$ is co-sofic in $\Gamma$.
Since $H\in\supp\left(\mu\right)$, Lemma \ref{lem:co-sofic-by-support}
implies that $H\in\overline{\Sub_{\findex}(\Gamma)}$.
\end{proof}
Specializing to $H=\left\{1\right\}$ in Proposition
\ref{prop:stable-is-fganpc}, we see that a sofic
P-stable group must be residually-finite, as proved by
Glebsky and Rivera (Theorem 2 of \cite{glebsky-rivera})
and by Arzhantseva and \Pau (Theorem 7.2(ii) of
\cite{arzhantseva-paunescu}).

By Proposition \ref{prop:stable-is-fganpc} and
Proposition \ref{prop:amenable-co-sofic},
if $\Gamma$ is amenable and P-stable,
then every almost-normal subgroup of $\Gamma$
is a limit in $\Sub(\Gamma)$ of finite-index subgroups.
If the converse is true as well (under the amenability assumption),
it would give a positive answer to the following question:
\begin{question}
\label{question:amenable-lerf-is-stable}
Is every amenable LERF group P-stable?
\end{question}
A related question was asked by Arzhantseva and \Pau
(see Conjecture 1.2 in \cite{arzhantseva-paunescu}):
$(\ast)$ Among finitely-presented
amenable groups, is P-stability equivalent to the following:
every normal subgroup of $\Gamma$ is profinitely-closed?
In fact, Conjecture 1.2 in \cite{arzhantseva-paunescu}
was stated differently, without assuming amenability, but $(\ast)$ is an equivalent
formulation under the amenability assumption
(see Theorem 7.2(iii) of \cite{arzhantseva-paunescu}).

\vspace{5mm}
Arzhantseva and \Pau asked whether every finitely-presented amenable
residually-finite group is P-stable (see the paragraph before Theorem 7.2 of \cite{arzhantseva-paunescu} and Theorem 7.2(iii) of the same paper).
We recall the construction of Abels's groups and show that they provide a
negative answer.
Fix a prime $p$. Abels's group (for the prime $p$) is
\[
A_{p}=\left\{ \left(\begin{array}{cccc}
1 & * & * & *\\
 & p^{m} & * & *\\
 &  & p^{n} & *\\
 &  &  & 1
\end{array}\right)\mid m,n\in\ZZ\right\} \leq\GL_{4}\left(\ZZ\left[\frac{1}{p}\right]\right)
\]

\begin{cor}
\label{cor:abels}Abels's group $A_{p}$ is finitely-presented,
amenable and residually-finite, but not P-stable.
\end{cor}

\begin{proof}
The group $A_p$ is amenable since it is solvable.
It is residually\hyp finite since it is finitely-generated and linear.
In fact, in \cite{abels}, Abels showed that $A_p$ is finitely-presented.
By Proposition \ref{prop:stable-is-fganpc}, to show that $A_p$
is not P-stable it suffices to exhibit a finitely-generated almost-normal
subgroup $H$ of $A_{p}$ which is not profinitely-closed. Note that
the center of $A_{p}$ is
\[
Z(A_p)=\left\{ \left(\begin{array}{cccc}
1 & 0 & 0 & x\\
 & 1 & 0 & 0\\
 &  & 1 & 0\\
 &  &  & 1
\end{array}\right)\mid x\in\ZZ\left[\frac{1}{p}\right]\right\} \cong\ZZ\left[\frac{1}{p}\right]
\]
 Let $H=\left\{ \left(\begin{array}{cccc}
1 & 0 & 0 & n\\
 & 1 & 0 & 0\\
 &  & 1 & 0\\
 &  &  & 1
\end{array}\right)\mid n\in\ZZ\right\} \cong\ZZ$. Since $H$ is cyclic and central, we are left with showing that
$H$ is not profinitely-closed in $A_{p}$. In general, if a group
$G$ is endowed with its profinite topology, then the subspace
topology on a subgroup $L\leq G$ is coarser (or equal)
to the profinite topology of $L$. Therefore, it suffices to
prove that $H$ is not closed in the profinite topology of $Z(A_p)$.
Consider the inclusion $\ZZ\subset\ZZ\left[\frac{1}{p}\right]$.
It suffices to show that the only finite quotient of $L=\ZZ\left[\frac{1}{p}\right]/\ZZ$
is the trivial group.
Note that for every $x\in L$, there is $n\in\NN$ such that $p^{n}x=0$, i.e. $L$ is a $p$-group, so a finite quotient of $L$ must be a finite $p$-group, and if this finite quotient is non-trivial, then $L$ has a quotient which is cyclic of order $p$. At the same time, $L$ is $p$-divisible, i.e. every element is a $p$-th multiple, and hence so is every quotient of $L$. But the cyclic group of order $p$ is not $p$-divisible.

\end{proof}

%
%
%

\end{document}